\documentclass[reqno,11pt,a4paper]{elsarticle}
\journal{Arxiv}

\usepackage{amssymb,amsmath,graphics,psfrag,graphicx,color,xcolor,subfigure,arydshln,fancyhdr}

\definecolor{lightgreen}{rgb}{0.22,0.50,0.25}
\definecolor{lightblue}{rgb}{0.22,0.45,0.70}
\usepackage{hyperref}
\hypersetup{colorlinks=true}

\numberwithin{equation}{section}
\numberwithin{table}{section}
\numberwithin{figure}{section}

\newtheorem{lemma}{Lemma}[section]
\newtheorem{theorem}{Theorem}[section]

\newcommand\ff{\boldsymbol{f}}

\newcommand\beps{\boldsymbol{\varepsilon}}
\newcommand\bbeta{\boldsymbol{\beta}}
\newcommand\bu{\boldsymbol{u}}
\newcommand\bv{\boldsymbol{v}}
\newcommand\bx{\boldsymbol{x}}
\newcommand\bz{\boldsymbol{z}}

\newcommand\R{\mathbb{R}}

\newcommand{\norm}[1]{\ensuremath{\left\|#1\right\|}}
\newcommand*\nnorm[1]{\vert\!\vert\!\vert #1 \vert\!\vert\!\vert}

\renewcommand\O{\Omega}
\newcommand\G{\Gamma}

\renewcommand\H{\mathrm{H}}
\renewcommand\L{\mathrm{L}}

\newcommand\LO{\L^2(\O)}
\newcommand\LOO{\L_{0}^2(\O)}
\newcommand\vdiv{\mathop{\mathrm{div}}\nolimits}
\newcommand\bdiv{\mathop{\mathbf{div}}\nolimits}

\newcommand\HCUO{\H_{0}^1(\O)}
\newcommand\HsO{\H^s(\O)}
\newcommand\HusO{\H^{1+s}(\O)}

\newcommand\bn{\boldsymbol{n}}

\newcommand\bw{\boldsymbol{w}}

\newcommand\curl{\mathop{\mathbf{curl}}\nolimits}

\renewcommand\P{\mathbb P}

\newcommand\cero{\boldsymbol{0}}

\newcommand\bomega{\boldsymbol{\omega}}
\newcommand\btheta{\boldsymbol{\theta}}
\newcommand\bnabla{\boldsymbol{\nabla}}

\newcommand{\ric}{}
\newcommand{\ruben}{}
\newcommand{\vero}{}

\newenvironment{proof}{\noindent{\it Proof.}}{\hfill$\square$}

\allowdisplaybreaks
\begin{document}
	\hypersetup{
		linkcolor=lightblue,
		urlcolor=lightblue,
		citecolor=lightblue
		}
	
\begin{frontmatter}

\title{On augmented finite element formulations for the Navier--Stokes equations with vorticity and variable viscosity}
\author[ubb,ci2ma]{Veronica Anaya}
\ead{vanaya@ubiobio.cl}
\author[ubb]{Rub\'en Caraballo}
\ead{ruben.caraballo1801@ubiobio.cl}
\author[monash,sechenov,unach]{Ricardo Ruiz-Baier}
\ead{ricardo.ruizbaier@monash.edu}
\author[uls]{H\'ector Torres}
\ead{htorres@userena.cl}

\address[ubb]{GIMNAP, Departamento de Matem\'atica, Universidad del B\'io-B\'io, Casilla 5-C, Concepci\'on, Chile.}
\address[ci2ma]{CI$^{\,2}\!$MA, Universidad de Concepci\' on, Chile.}	
\address[monash]{School of Mathematics, Monash University, 9 Rainforest walk, Melbourne, 3800 VIC, Australia.}
\address[sechenov]{World-Class Research Center ``Digital biodesign and personalized healthcare'', Sechenov First Moscow State Medical University, Moscow, Russia.}
\address[unach]{Universidad Adventista de Chile, Casilla 7-D, Chill\'an, Chile.}	
\address[uls]{Departamento de Matem\'aticas,  Universidad de La Serena, La Serena, Chile.}

\begin{abstract} We propose and analyse an augmented mixed finite element method for the Navier--Stokes equations written in terms of velocity, vorticity, and pressure with non-constant viscosity and  no-slip boundary conditions. The weak formulation includes least-squares terms arising from the constitutive equation and from the incompressibility condition. \ric{The theoretical and practical implications of using augmentation is discussed in detail. In addition, we use fixed--point} strategies to show the existence and uniqueness of  continuous and discrete solutions under the assumption of sufficiently small data. The method is  constructed using any compatible finite element pair for velocity and pressure as dictated by Stokes inf-sup stability, while for vorticity any generic discrete space (of arbitrary order) can be used. We establish optimal a priori error estimates. Finally, we provide a set of numerical tests in 2D and 3D illustrating the behaviour of the scheme as well as verifying the theoretical convergence rates.
\end{abstract}


	\end{frontmatter}

\section{Introduction}
\subsection{Scope} 
Incompressible flow equations in vorticity formulation play a critical role in describing rotational flows in a natural way. A diversity of formulations are available from the relevant literature. Of particular interest to us is their formulation in terms of velocity, vorticity (the curl of velocity), and pressure. These variables were used for Stokes and Navier--Stokes flows in the seminal papers \cite{chang90,bochev99,duan03,DSScmame03}, which analyse unique solvability and propose conforming discretisations. That form has leaded to a number of generalisations in Brinkman, Oseen, and Navier--Stokes equations and exploring different types of discretisation approaches including mixed finite element \cite{ACT,anaya13,anaya16,SSIMA15}, stabilised \cite{ACPT07,SScmame07}, least-squares \cite{anaya15b,bochev97,tsai05}, discontinuous Galerkin \cite{anaya19b}, adaptive, hybrid discontinuous Galerkin \cite{cockburn}, spectral \cite{abdel20,azaiez06,bernardi06,daikh17}, and preconditioned methods \cite{vassilevski14}. 
 
Similarly to these references, in the present work we are interested in the Navier--Stokes equations written in terms of velocity, vorticity, and pressure. However the present work considers the case of variable kinematic viscosity, for which the formulations above are not applicable since the viscous contribution to stress cannot be easily split in the usual rotation form (because, in general, $-\bdiv(\nu\bnabla\bu) \neq  \nu \curl(\curl\bu) - \nu \nabla(\vdiv\bu)$). Such an addition (and addressed for vorticity-based Stokes, Brinkman, and Oseen formulations in   \cite{ern98,anaya19,anaya24}) yields a non-symmetric variational form that can be augmented using residual terms from the constitutive and mass conservation equations. 
In the present treatment, additional terms appear due to the variable viscosity, which require a regularity assumption on the viscosity gradient. In \cite{anaya19} the kinematic viscosity $\nu$ is assumed in $W^{1,\infty}(\Omega)$, and in the Darcy--heat system analysed in \cite{bernardi18} the viscosity is assumed Lipschitz continuous. For our analysis it suffices to take $\nu  \in W^{1,r^\star}(\Omega)$, with $r^\star = \frac{2r}{r-2}$, and $r \in (2,6]$ for 3D and $r\in(2,\infty)$ for 2D. For example, $\nu  \in W^{1,3}(\Omega)$ in 2D or 3D. 

The extension to the Navier--Stokes case is carried out in this paper. The unique solvability in the constant viscosity case can be established as in  \cite{azaiez06,daikh17}, utilising antisymmetry of the convection term, appropriate continuity properties, smallness assumption on data, embeddings in separable Hilbert spaces, constructing sequences of finite-dimensional spaces, and   Brouwer's fixed--point theorem. 
 In our case the solvability analysis follows by a combination of Banach fixed--point theory using the velocity as fixed--point variable, and classical Babu\v{s}ka--Brezzi theory for saddle-point problems (by grouping together the velocity and vorticity unknowns). For this we have drawn inspiration from the analysis of Navier--Stokes--Darcy from the recent work \cite{gatica20}. The second aim of this paper is to construct a family of conforming discretisations. We can use simply Stokes-stable elements for velocity and pressure, while the vorticity can be approximated with arbitrary elements. For example, if choosing continuous or discontinuous piecewise polynomials of the same polynomial degree as pressure, we get an overall optimally convergent method with the same rates in all variables \ric{and in particular for vorticity, which would not be possible if one retrieves this field from a postprocess}. 

Of note, a number of differences with respect to rotational Oseen and Navier--Stokes equations are inherited from the augmentation. First,    the convective term is written in the usual way  for the standard velocity--pressure formulation, $(\bnabla\bu)\bu$, rather than in the rotational form $\curl\bu \times \bu$. This implies that we do not need to use the Bernoulli pressure (a nonlinear function of velocity module and kinematic pressure) to treat the momentum equation \cite{benzi12,layton09,olshanskii02,linke16}.  Secondly, as in \cite{anaya15b} one counts with  higher velocity regularity than the usual $\mathrm{H}(\vdiv,\Omega)$ typically achieved in the formulations in the former list of references (this feature  facilitates the analysis when manipulating the advecting term, but it comes at the expense of giving up the exact satisfaction of the divergence-free constraint at discrete level). Third, the augmentation permits us to obtain a higher convergence rate for  vorticity, and it allows us to easily impose no-slip velocity conditions. 
Another advantage of the present formulation is that it accommodates generalisations to non-isothermal systems and the coupling with other transport effects  in the very relevant case where viscosity depends on concentration or temperature \cite{patil82,payne99,rudi17}. \ric{This latter case (the coupling with additional effects) is a common example where viscosity varies but not necessarily the density. Other such examples include systems arising after linearisation of momentum equations of non-Newtonian fluids}.

\subsection{Outline} 
The contents of the paper have been organised as follows. Functional spaces and recurrent notation will be recalled and we will present the governing equations  in terms of velocity, vorticity and pressure, and  state the augmented formulation to the Navier--Stokes problem. The solvability analysis is presented in Section~\ref{sec:wellp}. The tools used therein are standard \ric{fixed--point} arguments. The Galerkin discretisation is presented in Section~\ref{sec:FEM}, where we also derive the stability analysis. Convergence rates for particular finite element subspaces including Taylor--Hood, Bernardi--Raugel, and MINI-elements are given in Section~\ref{sec:FE}. We visit several numerical tests illustrating the convergence of the proposed method under different scenarios (including cases not covered by our analysis) are reported in Section~\ref{sec:results}, and we close with a summary of our findings as well as concluding remarks laid out in Section~\ref{sec:concl}.

\subsection{Vorticity-based Navier--Stokes equations}\label{sec:model}
Let $\O$ be a  bounded domain of $\R^d$, $d=2,3$, with Lipschitz boundary $\G=\partial\O$. \ruben{For $r\in[1,\infty)$ we denote by $\L^r(\Omega)$ the usual Lebesgue and Sobolev spaces endowed with the norm $\Vert \cdot \Vert_{0,r,\Omega}:=\left(\int_{\Omega}|\cdot|^r \right)^{1/r}$, and} for any $s\geq 0$, the symbol $\norm{\cdot}_{s,\O}$ stands
for the norm of the Hilbertian Sobolev spaces $\HsO$ or
$\HsO^d$, with the usual convention $\H^0(\O):=\LO$. All through this paper $C$ will represent some absolute constant varying at each instance. 

Let us consider the Navier--Stokes problem with non-constant viscosity
modelling the steady-state flow of an incompressible viscous fluid.
The governing equations can be rewritten   using  the velocity $\bu$, the vorticity $\bomega$
and the pressure $p$  as follows (cf. \cite{ACPT07,amoura07,bochev97}):
Given a sufficiently smooth force density $\ff$,
we seek   $(\bu,\bomega,p)$ such that
\begin{subequations}
\begin{align}\label{eq:momentum_n}
\sigma\bu+\nu\curl\bomega+(\bu\cdot\bnabla)\bu-2\beps(\bu)\nabla\nu+\nabla p & = 
  \ff & \mbox{ in } \O, \\
  \bomega-\curl\bu&=\cero & \mbox{ in } \O,\label{eq:constitutive_n} \\ 
  \vdiv\bu & =  0 & \mbox{ in } \O,  \label{eq:mass_n}\\ 
  \bu & = \cero&    \mbox{ on } \G,\label{eq:bc1_n}\\
  (p,1)_{0,\O}&=0.&\label{eq:bc2_n}
 \end{align} 
 \end{subequations}

In the model {the kinematic viscosity of the fluid is assumed such that $\nu\in W^{1,r^\star}(\Omega)$  (with $r^\star$ made precise in \eqref{eq:r} below) and}  
$0<\nu_0\le\nu\le\nu_1$. 
Moreover, the coefficient $\sigma:\Omega \to \R$ satisfies \ruben{$0\leq\sigma_0\leq \sigma(x)\leq \sigma_1$}. {In the context of Navier--Stokes/Brinkman flows, it represents the inverse of the   permeability scaled with viscosity. In addition,} 
we will assume that $\ff\in L^2(\O)^d$.

The precise derivation and analysis of non-standard types of boundary conditions for the usual velocity--vorticity--pressure formulation is a delicate matter. We do not address it here and simply consider no-slip conditions for velocity everywhere on $\Gamma$, \ric{while non-homogeneous boundary conditions on velocity can be easily incorporated through a classical lifting argument.} 

{Equation \eqref{eq:momentum_n} results from using  the following, point-wisely satisfied, useful vector identity
\begin{equation*}
\curl(\curl\bv)=-\boldsymbol{\Delta}\bv+\nabla(\vdiv\bv),
\end{equation*}
and noting that the velocity $\bu$ is solenoidal, we have the chain of identities 
\begin{align*}
-2\bdiv(\nu\beps(\bu))& =-2\nu\bdiv(\beps(\bu))-2\beps(\bu)\nabla\nu   =-\nu\boldsymbol{\Delta}\bu-2\beps(\bu)\nabla\nu=\nu\curl(\curl\bu)-2\beps(\bu)\nabla\nu.\end{align*}}

\subsection{Variational formulation}
In this section, we propose a mixed variational formulation
of system \eqref{eq:momentum_n}-\eqref{eq:bc2_n}. 
%
In the sequel it will be useful to work with the product spaces 
 $X:=\HCUO^d\times\LO^{d(d-1)/2}$ and $H:=\HCUO^d\times\LO^{d(d-1)/2}\times\LOO$ equipped with the norms
\[
\Vert (\bv,\btheta) \Vert_{X}^2:=|\bv|_{1,\O}^2+\Vert \btheta \Vert_{0,\Omega}^2 \quad \text{and}\quad  \Vert (\bv,\btheta,q)  \Vert_{H}^2:=\Vert (\bv,\btheta) \Vert_{X}^2+\Vert q \Vert_{0,\Omega}^2,
\]
respectively. To derive a weak formulation we multiply by suitable test functions, integrate over the domain and apply the usual integration by parts as well as  
the following  version of Green's formula from e.g. \cite[Theorem~2.11]{gr-1986}: 
\begin{equation*}
(\curl\bomega,\bv)_{0,\O}=(\bomega,\curl\bv)_{0,\O}+\langle\bomega\times\bn,\bv\rangle_{0,\Gamma}.
\end{equation*}
We arrive at the following augmented variational formulation
for the Navier--Stokes problem \eqref{eq:momentum_n}--\eqref{eq:bc2_n}: 
Find $((\bu,\bomega),p)\in(\HCUO^d\times\LO^{d(d-1)/2})\times\LOO$ such that 
\begin{align}\label{probform2}
\nonumber a((\bu,\bomega),(\bv,\btheta))+N(\bu;\bu,\bv)+\ric{b(\bv,p)}=&\;F(\bv,\ric{\btheta}) & \forall(\bv,\btheta)\in\HCUO^d\times\LO^{d(d-1)/2},\\
\ric{b(\bu,q)}=&\;0 & \forall q\in\LOO,
\end{align}
where the bilinear{/trilinear} forms and the linear functional are defined as: 
 \begin{align}
\label{eq:bilinear}
a((\bu,\bomega),(\bv,\btheta))& := ( \sigma\bu ,\bv)_{0,\O} + \vero{( \nu\bomega ,\btheta)_{0,\O} }
+(\nu\bomega,\curl\bv)_{0,\O}-(\nu\btheta,\curl\bu)_{0,\O}   +\kappa_1(\curl\bu,\curl\bv)_{0,\O} \nonumber \\
&\qquad +\kappa_2(\vdiv\bu,\vdiv\bv)_{0,\O}  -\kappa_1(\bomega,\curl\bv )_{0,\O}
 -2(\beps(\bu)\nabla\nu,\bv)_{0,\O}+(\bomega,\nabla\nu\times\bv)_{0,\O},
 \\
\nonumber 
N(\bw;\bu,\bv)&:= ((\bw\cdot\nabla)\bu, \bv)_{0,\O},\quad 
  b(\bv,q)  :=-(q,\vdiv\bv)_{0,\O}, \quad F(\bv,\ric{\btheta}) :=(\ff,\bv)_{0,\O}, 
\end{align}
for all $(\bu,\bomega),(\bv,\btheta)\in\HCUO^d\times\LO^{d(d-1)/2}$, and
$q\in\LOO$.

\ric{The augmentation constant $\kappa_1$ enforces the constitutive equation for vorticity \eqref{eq:constitutive_n}. This parameter is needed to establish the coercivity of the upper-left block in the whole space, for a suitable saddle-point problem arising from the Oseen linearisation. We will see in Section~\ref{sec:augment} that there is a way to circumvent the use of this parameter by showing coercivity only on an appropriately defined kernel. In the discrete counterpart, assuming $\kappa_1=0$ will require as  additional constraint that the curl of the discrete velocity space is contained in the discrete vorticity space. 

On the other hand, note that the augmentation constant $\kappa_2$ (classical in grad-div type stabilisation approaches for incompressible flow) assists with improving the local divergence-free property \cite{jenkins14}. It is useful also in  our case, since our discrete velocity will be sought in an $\mathrm{H}^1-$conforming space. This is not required in, e.g., \cite{anaya19,anaya19b}, where H(div)-conforming velocities are used. For the analysis, this parameter is required in the discrete problem for the coercivity of the corresponding operators.} 


\section{Analysis of the {continuous} variational formulation}\label{sec:wellp}
\subsection{\ric{Preliminary results}}
To establish {well-posedness} of the nonlinear problem \eqref{probform2} we use a fixed--point {strategy}. First we require some preliminary notions. 
Consider   fixed real numbers   
\begin{equation}\label{eq:r}
r \in \begin{cases}  (2, \infty)   & \text{for $ d = 2 $},  \\
 (2,6]  &  \text{for $ d = 3 $}, \end{cases} \qquad \text{and} \quad r^*=\frac{2r}{r-2},\end{equation}
to be used when estimating terms of the form 
$( \bomega, \nabla \nu \times \bv)_{0,\O}$.

We continue by recalling the following Sobolev inequality {(see \cite{adams})},
\begin{equation}\label{embeddingh1}
\Vert \phi \Vert_{0,r,\Omega} \leq C_r \Vert \phi \Vert_{1,\Omega}, \qquad \forall \phi \in \H^1(\O).
\end{equation} 
Then, we see from \eqref{embeddingh1} that
\begin{equation}\label{embeddingh1d}
\Vert \bv \Vert_{0,r,\Omega} \leq C_r d^{\frac{r-2}{2r}}\Vert \bv \Vert_{1,\Omega}, \qquad \forall \bv \in \H^1(\O)^d,
\end{equation} 
which implies that 
\begin{equation}\label{betabounded}
\vert N(\bbeta;\bu,\bv) \vert=\left\vert ( (\bbeta \cdot \nabla)\bu,\bv)_{0,\O} \right\vert \leq C_4^2 d^{1/2}\Vert \bbeta\Vert_{1,\Omega} \vert \bu\vert_{1,\Omega} \Vert \bv\Vert_{1,\Omega} , \end{equation} 
for all $\bbeta,\bv,\bu \in \H^1(\O)^d$.
Next, using the identity 
\[
N(\bbeta;\bu,\bv)+N(\bbeta;\bv,\bu)=((\bbeta \cdot \nabla)\cdot\bu,\bv)_{0,\O}+((\bbeta \cdot \nabla)\bv,\bu)_{0,\O}=- (\vdiv \bbeta,\bu\cdot \bv)_{0,\O},
\]
we can deduce from \eqref{embeddingh1d} that
\begin{equation*}
\vert N(\bbeta;\bv,\bv)\vert=\left\vert( (\bbeta \cdot \nabla)\bv,\bv)_{0,\O} \right\vert \leq \frac{C_4^2 d^{1/2}}{2}\Vert \vdiv \bbeta\Vert_{0,\Omega} \Vert \bv\Vert_{1,\Omega}^2 , \  \forall \bbeta,\bv \in \H^1(\O)^d.
\end{equation*}
As a consequence \ric{of the usual inf-sup condition for the Stokes problem (cf. \cite{gr-1986})}, we have the following result.
\begin{lemma}
 \ric{There exists $\gamma>0$,
  such that}
\begin{equation*}
\ric{\sup_{0\ne\bv\in \H_0^1(\Omega)^d}\frac{\vert
  b(\bv,q)\vert}{|\bv|_{1,\Omega}}\ge \gamma\Vert
  q\Vert_{0,\O}\quad\forall q\in\LOO.}
\end{equation*}
\end{lemma}

Let $\varepsilon_1,\varepsilon_2,\varepsilon_3 \in \R^+$, then using Cauchy--Schwarz, Young and Poincaré inequality with constant $C_0$, we can deduce that 
\begin{align}
c(\nu \btheta,\curl \bv)_{0,\Omega} \leq &\dfrac{c \nu_1}{2\varepsilon_1}\Vert \curl \bv \Vert_{0,\Omega}^2+\dfrac{c\nu_1 \varepsilon_1}{2}\Vert \btheta \Vert_{0,\Omega}^2,\nonumber \\
2(\beps(\bv)\nabla \nu, \bv)_{0,\Omega} &\leq 2C_r C_0 d^{\frac{r-2}{2r}}\Vert \nabla \nu \Vert_{0,r^{*},\Omega} \vert \bv \vert_{1,\Omega}^2,\nonumber\\
(\btheta,\nabla \nu \times \bv)_{0,\Omega}&\leq \dfrac{2C_0^2C_r^2d^{\frac{r-2}{r}}\Vert \nabla \nu\Vert_{0,r^{*},\Omega}^2}{\varepsilon_2 \nu_0}\vert \bv \vert_{1,\Omega}^2+\dfrac{\varepsilon_2 \nu_0}{2}\Vert \btheta \Vert_{0,\Omega},\label{bounds00}\\
(\curl\bv,\nabla \nu \times \bv)_{0,\Omega}&\leq \dfrac{2C_0^2C_r^2d^{\frac{r-2}{r}}\Vert \nabla \nu\Vert_{0,r^{*},\Omega}^2}{\varepsilon_3 \nu_0}\vert \bv \vert_{1,\Omega}^2+\dfrac{\varepsilon_3 \nu_0}{2}\Vert \curl \bv \Vert_{0,\Omega},\nonumber\\
N(\bbeta;\bv,\bv) & \leq \dfrac{C_0^2C_4^2d^{\frac{1}{2}}}{2}\Vert \vdiv \bbeta \Vert_{0,\Omega} \vert \bv \vert_{1,\Omega}^2.\nonumber
\end{align}
\subsection{\ric{Analysis for the case $0<\kappa_1<\frac{2}{3}\nu_0$ and $\kappa_2>0$}}\label{sec:k1>0}

Next, and similarly \ric{to} \cite{anaya24}, we can readily state the following collection of results.
\begin{lemma}\label{lem-elip}
Let $\kappa_1,\kappa_2>0$ with $0<\kappa_1<\frac{2}{3}\nu_0$, and 
assume that $\bbeta\in \H^1(\Omega)^d$ and
\begin{equation*}
C_0^2C_r^2 d^{\frac{r-2}{r}}\Vert \nabla \nu \Vert_{0,r^*,\Omega}^2\left\{ \frac{1}{\kappa}+\frac{3}{\nu_0}\right\}+\frac{C_4^2 d^{1/2}}{2}\Vert \vdiv \bbeta\Vert_{0,\Omega} <\min\left\{ \frac{\kappa_2}{2}, \kappa_1-\frac{3\kappa_1^2}{4\nu_0}\right\},
\end{equation*}
where $\kappa=\min\{\kappa_1,\kappa_2\}$. Then, the bilinear/trilinear forms defined in \eqref{eq:bilinear} are bounded 
\begin{equation}\label{boundedoperators}
\Vert a \Vert\leq A_0 A_1, \quad \Vert b\Vert\leq 1, \quad \Vert F\Vert\leq \Vert \ff\Vert_{0,\Omega}, \quad \left\Vert a +N(\bbeta;\cdot,\cdot) \right\Vert  \leq B_0(\bbeta)B_1(\bbeta),
\end{equation}

with
\begin{align*}
A_0&\ric{:=}\max \{\sigma_1C_0^2,\nu_1+\kappa_1,\max\{\kappa_2,\nu_1+\kappa_1\}+2C_0 C_r d^{\frac{r-2}{2r}}\Vert \nabla \nu \Vert_{0,r^*,\Omega} \}^{1/2},\\
A_1&\ric{:=}\max\{\nu_1+2C_0C_r d^{\frac{r-2}{2r}}\Vert \nabla \nu \Vert_{0,r^*,\Omega}, \max\{\sigma_1C_0^2,\kappa_2,\nu_1+2\kappa_1 \}+4C_0C_r d^{\frac{r-2}{2r}}\Vert \nabla \nu \Vert_{0,r^*,\Omega}\}^{1/2},\\
B_0(\bbeta)&\ric{:=}\max \{\sigma_1C_0^2,\nu_1+\kappa_1, 
\max\{\kappa_2,\nu_1+\kappa_1\}+2C_0C_r d^{\frac{r-2}{2r}}\Vert \nabla \nu \Vert_{0,r^*,\Omega} +C_0^2C_4^2d^{1/2}| \bbeta |_{1,\Omega}\}^{1/2},\\
B_1(\bbeta)&\ric{:=}\max\{\nu_1+2C_0C_r d^{\frac{r-2}{2r}}\Vert \nabla \nu \Vert_{0,r^*,\Omega},\max\{\sigma_1C_0^2,\kappa_2,\nu_1+2\kappa_1 \}+4C_0C_r d^{\frac{r-2}{2r}}\Vert \nabla \nu \Vert_{0,r^*,\Omega}\\
& \qquad +C_0^2C_4^2d^{1/2}\Vert \bbeta \Vert_{1,\Omega}\}^{1/2}.
\end{align*}
Moreover, for a $\bbeta$ given, the bilinear form $a(\cdot,\cdot)+N(\bbeta;\cdot,\cdot)$ is $X$-elliptic 
\[
a((\bv,\btheta),(\bv,\btheta))+N(\bbeta;\bv,\bv)\geq \widehat{\alpha}(\bbeta) \Vert (\bv,\btheta) \Vert_X^2 \qquad \forall (\bv,\btheta)\in X,
\]
where $\widehat{\alpha}(\bbeta):= \widehat{\alpha}(\nu,\kappa_1,\kappa_2,\bbeta)=\min\left\{ \frac{\nu_0}{3}, \lambda \right\}$ with  
\[
\lambda:=\min\left\{ \frac{\kappa_2}{2}, \kappa_1-\frac{3\kappa_1^2}{4\nu_0}\right\}-C_0^2C_r^2 d^{\frac{r-2}{r}}\Vert \nabla \nu \Vert_{0,r^*,\Omega}^2\left\{ \frac{1}{\kappa}+\frac{3}{\nu_0}\right\}-\frac{C_4^2 d^{1/2}}{2}\Vert \vdiv \bbeta\Vert_{0,\Omega}.
\]
\end{lemma}
\ric{The continuity and coercivity properties established in Lemma~\ref{lem-elip} imply the solvability of a  linearised problem:}
\begin{lemma}\label{lem-cont}
There exists $(\bu,\bomega,p)\in \HCUO^d\times\LO^{d(d-1)/2}\times\LOO$ such that 
\begin{align}\label{Navi-Ose}
\nonumber  a((\bu,\bomega),(\bv,\btheta))+N(\bbeta;\bu, \bv)+b(\bv,p)&=\;F(\bv,
\btheta) 
  &\forall(\bv,\btheta)\in\HCUO^d\times\LO^{d(d-1)/2},\\
  b(\bu,q)&=\;0 & \forall q\in\LOO.
\end{align}
In addition, the solution satisfies the following continuous dependence on  data 
\[ \Vert (\bu,\bomega) \Vert_X \leq \frac{\Vert \ff \Vert_{0,\Omega}}{\widehat{\alpha}(\bbeta)}, \quad \Vert p\Vert_{0,\Omega}\leq \gamma^{-1}\left( 1+\frac{\left\Vert a +N(\bbeta;\cdot,\cdot) \right\Vert }{\widehat{\alpha}(\bbeta)}\right) \Vert \ff \Vert_{0,\Omega}.\]
\end{lemma}
In turn, Lemmas \ref{lem-elip} and \ref{lem-cont} will be used in showing the existence of a suitably defined fixed point, stated as follows.  
\begin{lemma}\label{punto_fijo}
Let $\kappa_2>0$ and $0<\kappa_1<\frac{2}{3}\nu_0$, and suppose that
\begin{equation}\label{eq:yt7}
C_0^2C_r^2 d^{\frac{r-2}{r}}\Vert \nabla \nu \Vert_{0,r^*,\Omega}^2\left\{ \frac{1}{\kappa}+\frac{3}{\nu_0}\right\} < \min\left\{ \frac{\kappa_2}{2}, \kappa_1-\frac{3\kappa_1^2}{4\nu_0}\right\}.
\end{equation}
Consider the set
\[
\mathcal{O}:=\{\bu\in \HCUO^d: |\bu|_{1,\O}\leq \delta\},
\]
defined by choosing $0<\delta<\frac{{ \overline{\alpha}}(\nu) }{C_0^2C_4^2 d^{1/2}}$  with 
\begin{equation}\label{alpha1}
\alpha(\nu):=\alpha(\nu,\kappa_1,\kappa_2)=\min\left\{ \frac{\kappa_2}{2}, \kappa_1-\frac{3\kappa_1^2}{4\nu_0}\right\}- C_0^2C_r^2 d^{\frac{r-2}{r}}\Vert \nabla \nu \Vert_{0,r^*,\Omega}^2\left\{ \frac{1}{\kappa}+\frac{3}{\nu_0}\right\},
\end{equation}
and $\overline{\alpha}(\nu):=\min\left\{ \frac{\nu_0}{3},\alpha\right\}$. If 
\begin{equation}\label{eq:yt8}
\Vert \ff \Vert_{0,\Omega} <\frac{1}{2}{ \overline{\alpha}}(\nu)  \delta,\end{equation}
then, the operator {$\mathcal{J}:\mathcal{O}\to \mathcal{O}$, with $\bbeta \mapsto \mathcal{J}(\bbeta)=\bu$ (and where $\bu$ is the velocity solution of the Oseen problem \eqref{Navi-Ose}, for a $\bbeta$ given), has} a unique fixed point in $\mathcal{O}$. 
\end{lemma}
\begin{proof}
From the {assumptions \eqref{eq:yt7}-\eqref{eq:yt8} it follows that} $\mathcal{J}$ is {well-defined} in $\mathcal{O}$.
Next, we note that for all $\bbeta \in \mathcal{O}$, {the following chain of bounds holds}  
\[
\frac{\Vert \ff\Vert_{0,\Omega}}{\delta}+\frac{C_0^2C_4^2 d^{1/2}}{2}\Vert \vdiv \bbeta \Vert_{0,\Omega}\leq \frac{\Vert \ff\Vert_{0,\Omega}}{\delta}+\frac{C_0^2C_4^2 d^{1/2}}{2}| \bbeta|_{1,\Omega}\leq \frac{1}{2}\alpha(\nu) +\frac{C_4^2 d^{1/2}\delta}{2}<\alpha(\nu),
\]
and consequently 
\[
|\mathcal{J}(\bbeta)|_{1,\Omega}\leq \frac{\Vert \ff\Vert_{0,\Omega}}{\widehat{\alpha}(\bbeta)}< \delta,
\]
which implies that $\mathcal{J}(\mathcal{O})\subset \mathcal{O}$.

Now, we have to prove that $\mathcal{J}$ is a contraction. In fact, for each $i\in \{1,2\}$ suppose that we have {$\mathcal{J}(\bbeta_i)=\bu_i$. Therefore},  for each triplet $(\bv,\btheta,q)\in \HCUO^d\times\LO^{d(d-1)/2}\times\LOO$ we have that
\begin{subequations}
\begin{align}\label{E1}
a((\bu_1-\bu_2,\bomega_1-\bomega_2),(\bv,\btheta))+b(\bv,p_1-p_2)&=N(\bbeta_2;\bu_2,\bv)-N(\bbeta_1;\bu_1,\bv),\\
\label{E2}
b(\bu_1-\bu_2,q)&=0.
\end{align}\end{subequations}
On the other hand, it is evident that 
$$
-N(\bbeta_1;\bu_1,\bv)+N(\bbeta_2;\bu_2,\bv)=N(\bbeta_2-\bbeta_1;\bu_1,\bv)+N(\bbeta_2;\bu_2-\bu_1,\bv),
$$
and using the triangle inequality {together with the bound} \eqref{betabounded} we obtain
\begin{align*}
  -N(\bbeta_1;\bu_1,\bv)+N(\bbeta_2;\bu_2,\bv) 
  \leq C_0^2C_4^2 d^{1/2}(| \bbeta_1-\bbeta_2|_{1,\O}| \bu_1|_{1,\O}| \bv |_{1,\O}+| \bbeta_2|_{1,\O}|\bu_2-\bu_1|_{1,\O}| \bv |_{1,\O}).
\end{align*} 
{In turn, t}aking $\bv=\bu_2-\bu_1$, $\btheta=\bomega_2-\bomega_1$ and using the equalities \eqref{E1}-\eqref{E2}, we can assert that 
\begin{align*}
{ \overline{\alpha}}(\nu)  \Vert (\bu_1-\bu_2,\bomega_1-\bomega_2)\Vert_{{X}}^2   \leq C_0^2C_4^2 d^{1/2}(| \bbeta_1-\bbeta_2|_{1,\O}| \bu_1|_{1,\O}| \bu_1-\bu_2 |_{1,\O}+| \bbeta_2|_{1,\O}|\bu_2-\bu_1|_{1,\O}^2),
\end{align*}
and from the above inequality and the definition of $\mathcal{O}$, we get
$$
{ \overline{\alpha}}(\nu)  \Vert (\bu_1-\bu_2,\bomega_1-\bomega_2)\Vert_{{X}}^2 \leq C_0^2C_4^2 d^{1/2}(\delta | \bbeta_1-\bbeta|_{1,\O}| \bu_1-\bu_2|_{1,\O}+\delta|\bu_2-\bu_1|_{1,\O}^2).
$$
{Recalling the choice of $\delta$ we readily see that} $\frac{ C_0^2C_4^2 d^{1/2} \delta { \overline{\alpha}(\nu)}^{-1}}{1- C_0^2C_4^2 d^{1/2} \delta { \overline{\alpha}(\nu)}^{-1}}<1$ and using the bound
$$
|\bu_2-\bu_1|_{1,\O}\leq \Vert (\bu_1-\bu_2,\bomega_1-\bomega_2)\Vert_{{X}},
$$
{we can conclude that} 
$$
 |\bu_2-\bu_1|_{1,\O} \leq \frac{ C_0^2C_4^2 d^{1/2} \delta { \overline{\alpha}} ^{-1}}{1- C_0^2C_4^2 d^{1/2} \delta { \overline{\alpha}} ^{-1}}|\bbeta_1-\bbeta_2|_{1,\O}.
$$
Then $\mathcal{J}$ is a contraction and,  \ric{thanks to Banach's fixed--point theorem, it has a unique fixed point $\bu\in \mathcal{O}$}. 
\end{proof}

\subsection{\ric{Analysis for the case $\kappa_1=0$ and $\kappa_2>0$}}\label{sec:augment}
\ric{An alternative proof can be still conducted if $\kappa_1=0$. The analysis from Section~\ref{sec:k1>0} is in such a case modified as follows. }

\ruben{
First, the counterpart of Lemma \ref{lem-elip} reads 
\begin{lemma}\label{lem-elip2}
If $\varepsilon_2\in (0,2)$ in \eqref{bounds00}, $\kappa_2>0$ and 
$$
\dfrac{\nu_0^3}{2\nu_1^2}\left( 1-\dfrac{\varepsilon_2}{2} \right)> \mathcal{C}(\nu,\bbeta), \text{   } \dfrac{9}{32}\dfrac{\nu_0^2}{\nu_1^2}\left( 1-\dfrac{\varepsilon_2}{2}\right)>\widehat{\mathcal{C}}(\nu,\bbeta),
$$
then \eqref{boundedoperators} holds with the specifications
\begin{align*}
A_0&:=\left( \sigma_1C_0^2+2\max\{ \nu_1,\kappa_2\}+\dfrac{\varepsilon_2^{1/2}\nu_0^2}{2C_0 \nu_1}\left( 1-\dfrac{\varepsilon_2}{2}\right)^{1/2} \right)^{1/2},\\
A_1&:=\left( \sigma_1C_0^2+2\max\{\nu_1,\kappa_2 \}+\dfrac{\varepsilon_2^{1/2}\nu_0^2}{ \nu_1}\left( 1-\dfrac{\varepsilon_2}{2}\right)^{1/2} \right)^{1/2},\\
B_0(\nu)&:=(A_0^2+C_0^2C_4^2d^{1/2}| \bbeta |_{1,\Omega})^{1/2}, \quad B_1(\nu):=(A_1^2+C_0^2C_4^2d^{1/2}| \bbeta |_{1,\Omega})^{1/2}.
\end{align*}
Moreover, and instead $X$-ellipticity, we now have that there exist $\widehat{\alpha}(\bbeta)$ such that
$$
\sup_{0\ne(\widehat{\bv},\widehat{\btheta})\in \mathrm{Ker}(b)\times \L^2(\Omega)}\dfrac{a((\bv,\btheta),(\widehat{\bv},\widehat{\btheta}))+N(\bbeta;\bv,\widehat{\bv})}{\Vert(\widehat{\bv},\widehat{\btheta}) \Vert_{X}}\geq \widehat{\alpha}(\bbeta) \Vert(\bv,\btheta) \Vert_{X},
$$
and 
$$
\sup_{0\ne(\widehat{\bv},\widehat{\btheta})\in \mathrm{Ker}(b)\times \L^2(\Omega)} [a((\widehat{\bv},\widehat{\btheta}),(\bv,\btheta))+N(\bbeta;\widehat{\bv},\bv)]>0,
$$
where $\mathrm{Ker}(b):=\{\bv \in  \H_0^1(\O)^d: \text{ }\vdiv \bv=0 \text{ in } \L^2(\O)\}$.
\end{lemma}
\begin{proof}
For $\bv$ with $\vdiv \bv=0$ and $c\in \R$ we have
\begin{align}
a((\bv,\btheta),(\bv, \btheta-c \curl \bv))+N(\bbeta;\bv,\bv)& \geq \sigma_0 \Vert \bv \Vert_{0,\Omega}^2 + \nu_0 \Vert \btheta \Vert_{0,\Omega}^2-c(\nu \btheta,\curl \bv)_{0,\Omega}+c(\nu \curl \bv, \curl \bv)_{0,\Omega}\nonumber \\
& \qquad +\kappa_2 \Vert \vdiv \bv \Vert_{0,\Omega}^2-2(\beps(\bv)\nabla \nu, \bv)_{0,\Omega}+(\btheta,\nabla \nu \times \bv)_{0,\Omega}+N(\bbeta;\bv,\bv)\nonumber\\
&\geq\sigma_0 \Vert \bv \Vert_{0,\Omega}^2+\lambda_1(\nu) \Vert \btheta \Vert_{0,\Omega}^2+\lambda_2(\nu,\bbeta) \Vert \curl \bv \Vert_{0,\Omega}^2 \nonumber\\
&= \sigma_0 \Vert \bv \Vert_{0,\Omega}^2+\lambda_1 (\nu)\Vert \btheta \Vert_{0,\Omega}^2+\lambda_2 (\nu,\bbeta)|  \bv |_{1,\Omega}^2,\label{curl1}
\end{align}
where 
$$
\lambda_1(\nu):=\nu_0\left( 1-\dfrac{\varepsilon_2}{2}\right)-\dfrac{c\nu_1 \varepsilon_1}{2}, \text{ }
\lambda_2(\nu,\bbeta):= c\nu_0-\dfrac{c \nu_1}{2\varepsilon_1}-\mathcal{C}(\nu,\bbeta).
$$
It suffices to choose $\varepsilon_2 \in (0,2)$, $\varepsilon_1=\dfrac{\nu_1}{\nu_0}$, $c=\dfrac{\nu_0^2}{\nu_1^2}\left( 1-\dfrac{\varepsilon_2}{2}\right)$ in \eqref{bounds00} and to assume that $\dfrac{\nu_0^3}{2\nu_1^2}\left( 1-\dfrac{\varepsilon_2}{2}\right)>\mathcal{C}(\nu,\bbeta)$, to deduce that  $\lambda_1(\nu),\lambda_2(\nu,\bbeta)\in \R^{+}$.

Then we can infer that
$$
\sup_{0\ne(\widehat{\bv},\widehat{\btheta})\in \mathrm{Ker}(b)\times \L^2(\Omega)}\dfrac{a((\bv,\btheta),(\widehat{\bv},\widehat{\btheta}))+N(\bbeta;\bv,\widehat{\bv})}{\Vert(\widehat{\bv},\widehat{\btheta}) \Vert_{X}}\geq \widehat{\alpha}(\bbeta) \Vert(\bv,\btheta) \Vert_{X},
$$
where
$$
\widehat{\alpha}(\bbeta):=\dfrac{1}{2} \min\left\{\nu_0\left(1-\dfrac{\varepsilon_2}{2}\right),\dfrac{\nu_0^3}{\nu_1^2}\left(1-\dfrac{\varepsilon_2}{2}\right)-2\mathcal{C}(\nu,\bbeta)\right\}.
$$
Furthermore we have, again for  $\bv$ with $\vdiv \bv=0$ and $c\in \R$, that 
\begin{align}
a((\bv, \btheta+c \curl \bv),(\bv,\btheta))+N(\bbeta;\bv,\bv)& \geq \sigma_0 \Vert \bv \Vert_{0,\Omega}^2 + \nu_0 \Vert \btheta \Vert_{0,\Omega}^2+c(\nu \curl \bv,\btheta)_{0,\Omega}\nonumber \\
& \qquad +c(\nu \curl \bv, \curl \bv)_{0,\Omega}+\kappa_2 \Vert \vdiv \bv \Vert_{0,\Omega}^2-2(\beps(\bv)\nabla \nu, \bv)_{0,\Omega}\nonumber\\
&\qquad+(\btheta,\nabla \nu \times \bv)_{0,\Omega}+c(\curl \bv,\nabla \nu \times \bv)_{0,\Omega}+N(\bbeta;\bv,\bv)\nonumber\\
&\geq\sigma_0 \Vert \bv \Vert_{0,\Omega}^2+\lambda_1(\nu) \Vert \btheta \Vert_{0,\Omega}^2+\lambda_3 (\nu,\bbeta)\Vert \curl \bv \Vert_{0,\Omega}^2 \nonumber\\
&= \sigma_0 \Vert \bv \Vert_{0,\Omega}^2+\lambda_1 (\nu)\Vert \btheta \Vert_{0,\Omega}^2+\lambda_3(\nu,\bbeta)|  \bv |_{1,\Omega}^2,\label{curl2}
\end{align}
where
$$
\lambda_3(\nu,\bbeta):= c\nu_0-\dfrac{\varepsilon_3\nu_0}{2}-\dfrac{c \nu_1}{2\varepsilon_1}-\widehat{\mathcal{C}}(\nu,\bbeta).
$$

Taking again $\varepsilon_2 \in (0,2)$, while choosing $\varepsilon_1=\dfrac{\nu_1}{\nu_0}$, $\varepsilon_3=\dfrac{3}{8}\dfrac{\nu_0^2}{\nu_1^2}\left( 1-\dfrac{\varepsilon_2}{2}\right)$, $ c=\dfrac{3}{4}\dfrac{\nu_0^2}{\nu_1^2}\left(1-\dfrac{\varepsilon_2}{2}\right)$, in \eqref{bounds00}
and assuming that 
$$
\dfrac{9}{32}\dfrac{\nu_0^2}{\nu_1^2}\left( 1-\dfrac{\varepsilon_2}{2}\right)>\widehat{\mathcal{C}}(\nu,\bbeta),
$$ 
we can conclude that $\lambda_1(\nu),\lambda_3(\nu,\bbeta)\in \R^{+}$, and therefore
$$
\sup_{0\ne(\widehat{\bv},\widehat{\btheta})\in \mathrm{Ker}(b)\times \L^2(\Omega)} [a((\widehat{\bv},\widehat{\btheta}),(\bv,\btheta))+N(\bbeta;\widehat{\bv},\bv)]>0.
$$
For the discrete case the lemma is analogous. We just have to add the condition $\kappa_2>\widehat{\mathcal{C}}(\nu,\bbeta)$ and modify $\widehat{\alpha}(\bbeta)$ as
$$
\widehat{\alpha}(\bbeta):=\dfrac{1}{2} \min\left\{\nu_0\left(1-\dfrac{\varepsilon_2}{2}\right),\dfrac{\nu_0^3}{\nu_1^2}\left(1-\dfrac{\varepsilon_2}{2}\right)-2\mathcal{C}(\nu,\bbeta),\kappa_2-\mathcal{C}(\nu,\bbeta)\right\}.
$$
\end{proof}

\begin{lemma}\label{punto_fijo2}
Let $\varepsilon_2\in (0,2)$ in \eqref{bounds00}, $\kappa_2>0$ and assume that  
\begin{equation}\label{asummtion1}
\dfrac{\nu_0^3}{2\nu_1^2}\left( 1-\dfrac{\varepsilon_2}{2} \right)> \mathcal{D}(\nu), \text{   } \dfrac{9}{32}\dfrac{\nu_0^2}{\nu_1^2}\left( 1-\dfrac{\varepsilon_2}{2}\right)>\widehat{\mathcal{D}}(\nu).
\end{equation}
Consider the set
$$
\mathcal{O}:=\{ \bu \in H_0^1(\Omega)^d: \quad |\bu|_{1,\Omega}\leq \delta \},
$$
defined by choosing $0<\delta <\dfrac{\overline{\alpha}(\nu)}{C_0^2C_4^2d^{1/2}}$ with
\begin{equation}\label{alpha2}
\alpha(\nu):=\dfrac{\nu_0^3}{2\nu_1^2}\left( 1-\dfrac{\varepsilon_2}{2} \right)- \mathcal{D}(\nu) \text{ and } \overline{\alpha}(\nu):=\min \left\{ \dfrac{\nu_0}{2}\left(1-\dfrac{\varepsilon_2}{2}\right),\alpha(\nu)\right\}.
\end{equation}
If 
\begin{equation}\label{asummtion2}
\Vert \ff \Vert_{0,\Omega}\leq \dfrac{1}{2}\overline{\alpha}(\nu)\delta,
\end{equation}
then, the operator $\mathcal{J}:\mathcal{O} \rightarrow \mathcal{O}$, with $\bbeta \mapsto \mathcal{J}(\bbeta)=\bu$ (and where $\bu$ is the velocity solution of the Oseen problem
\eqref{Navi-Ose}, for a $\bbeta$ given), has a unique fixed point in $\mathcal{O}$.
\end{lemma}
\begin{proof}
From the assumptions \eqref{asummtion1}-\eqref{asummtion2} it follows that $\mathcal{J}$ is well-defined in $\mathcal{O}$.
Next, we note that for all $\bbeta\in \mathcal{O}$, the following chain of bounds holds
\[
\frac{\Vert \ff\Vert_{0,\Omega}}{\delta}+\frac{C_4^2 d^{1/2}}{2}\Vert \vdiv \bbeta \Vert_{0,\Omega}\leq \frac{1}{2}\alpha(\nu) +\frac{C_4^2 d^{1/2}\delta}{2}<\alpha(\nu),
\]
and consequently 
\[
\nnorm{\mathcal{J}(\bbeta)}_{1,\Omega}\leq \frac{\Vert \ff\Vert_{0,\Omega}}{\widehat{\alpha}(\bbeta)}< \delta,
\]
which implies that $\mathcal{J}(\mathcal{O})\subset \mathcal{O}$.

To show  that $\mathcal{J}$ is a contraction we suppose that we have $\mathcal{J}(\bbeta_i)=\bu_i$, $i=1,2$. Then for each $(\bv,\btheta,q)\in \HCUO^d\times\LO^{d(d-1)/2}\times\LOO$ we have that
\begin{subequations}
\begin{align}\label{E12}
a((\bu_1-\bu_2,\bomega_1-\bomega_2),(\bv,\btheta))+b(\bv,p_1-p_2)&=N(\bbeta_2;\bu_2,\bv)-N(\bbeta_1;\bu_1,\bv),\\
\label{E22}
b(\bu_1-\bu_2,q)&=0.
\end{align}\end{subequations}

On the other hand, it is evident that 
$$
-N(\bbeta_1;\bu_1,\bv)+N(\bbeta_2;\bu_2,\bv)=N(\bbeta_2-\bbeta_1;\bu_1,\bv)+N(\bbeta_2;\bu_2-\bu_1,\bv),
$$
and using the triangle inequality {together with the bound} \eqref{betabounded} we obtain
\begin{align*}
  -N(\bbeta_1;\bu_1,\bv)+N(\bbeta_2;\bu_2,\bv) 
  \leq C_0^2C_4^2 d^{1/2}(| \bbeta_1-\bbeta_2|_{1,\O}| \bu_1|_{1,\O}| \bv |_{1,\O}+| \bbeta_2|_{1,\O}|\bu_2-\bu_1|_{1,\O}| \bv |_{1,\O}).
\end{align*} 
{In turn, t}aking $\bv=\bu_2-\bu_1$, $\btheta=\bomega_2-\bomega_1+c\curl(\bu_2-\bu_1)$ and using the equalities \eqref{E12}-\eqref{E22}, we can assert that 
\begin{align*}
{ \overline{\alpha}}(\nu)  \Vert (\bu_1-\bu_2,\bomega_1-\bomega_2)\Vert^2   \leq C_0^2C_4^2 d^{1/2}(| \bbeta_1-\bbeta_2|_{1,\O}| \bu_1|_{1,\O}| \bu_1-\bu_2 |_{1,\O}+| \bbeta_2|_{1,\O}|\bu_2-\bu_1|_{1,\O}^2),
\end{align*}
and from the above inequality and the definition of $\mathcal{O}$, we get
$$
{ \overline{\alpha}(\nu)}  \Vert (\bu_1-\bu_2,\bomega_1-\bomega_2)\Vert^2 \leq C_0^2C_4^2 d^{1/2}(\delta | \bbeta_1-\bbeta_2|_{1,\O}| \bu_1-\bu_2|_{1,\O}+\delta |\bu_2-\bu_1|_{1,\O}^2).
$$
{Recalling the choice of $\delta$ we readily see that} $\frac{ C_0^2C_4^2 d^{1/2} \delta { \overline{\alpha}(\nu)}^{-1}}{1-C_0^2 C_4^2 d^{1/2} \delta { \overline{\alpha}(\nu)}^{-1}}<1$ and using the bound
$$
|\bu_2-\bu_1|_{1,\O}\leq \Vert (\bu_1-\bu_2,\bomega_1-\bomega_2)\Vert_{X},
$$
{we can conclude that} 
$$
 |\bu_2-\bu_1|_{1,\O} \leq \frac{ C_0^2C_4^2 d^{1/2} \delta { \overline{\alpha}(\nu)}^{-1}}{1-C_0^2 C_4^2 d^{1/2} \delta { \overline{\alpha}(\nu)}^{-1}}|\bbeta_1-\bbeta_2|_{1,\O}.
$$
Then $\mathcal{J}$ is a contraction and,  thanks to Banach's fixed--point theorem, it has a unique fixed point $\bu\in \mathcal{O}$. 
\end{proof}

Owing to the previous results, we have the unique solvability of \eqref{probform2},  stated in the following theorem. 
\begin{theorem}
Let $\ff \in \LO^d$, and proceed under the assumptions of Lemma~\ref{punto_fijo} if $\kappa_1>0$ or of Lemma~\ref{punto_fijo2} if $\kappa_1=0$. Then the operator $\mathcal{J}$ has a unique fixed point $\bu \in \mathcal{O}$. Equivalently, problem \eqref{probform2} has a unique solution
$(\bu,\bomega,p) \in \HCUO^d\times\LO^{d(d-1)/2}\times\LOO$ with $\bu \in \mathcal{O}.$
\end{theorem}
}


\section{Galerkin discretisation and error estimates} \label{sec:FEM}
In this section we formulate a discrete problem associated {with} \eqref{probform2} taking generic finite dimensional subspaces that yield unique solvability and a C\'ea estimate. We also derive a priori error estimates, and provide the {rate of convergence expected for some examples of well-known finite element families.} 

\subsection{{Preliminaries and unique solvability}}
Let $\{\mathcal{ T}_{h}(\Omega)\}_{h>0}$ be a shape-regular
family of partitions of the polygon/polyhedron 
$\bar\O$, by triangles / tetrahedra $T$ of diameter $h_T$,
with mesh size $h:=\max\{h_T:\; T\in\mathcal{T}_{h}(\O)\}$.
Given an integer $k\ge0$ and a subset
$S$ of $\R^d$, the symbol $\mathbb{P}_k(S)$ denotes the space of polynomial functions
defined in $S$ of total degree up to $k$.
We consider generic finite dimensional subspaces
\ric{$V_h\subseteq\HCUO^d$ 
  and $Q_h\subseteq\LOO$ such that}
\begin{equation}\label{inf-sup-d}
\sup_{\ric{0\ne \bv_h \in V_h}}\frac{\vert
b(\bv_h,q_h)\vert}{\ric{|\bv_h|_{1,\Omega}}}
\ge \gamma_0\Vert q_h\Vert_{0,\O}\quad\forall q_h\in Q_h,
\end{equation}
where $\gamma>0$ is independent of $h$. {Relation \eqref{inf-sup-d} 
is satisfied for inf-sup stable pairs for
the Stokes problem}. In turn, 
the discrete space $W_h\subseteq \LO^{d(d-1)/2}$
for  vorticity can be taken as a continuous
or discontinuous polynomial space; both options will be addressed. 
 The Galerkin  problem reads: 
 Find $(\bu_h,\bomega_h,p_h)\in V_h\times W_h\times Q_h$ such that 
\begin{align}\label{probform2d}
\nonumber a((\bu_h,\bomega_h),(\bv_h,\btheta_h))+N(\bu_h;\bu_h,\bv_h) +b(\bv_h,p_h)=&\;F(\bv_h,
\btheta_h) & 
 \forall(\bv_h,\btheta_h)&\in V_h\times W_h,\\
  b(\bu_h,q_h)=&\;0 &  \forall q\in Q_h.
\end{align}

As in the  continuous analysis,  problem \eqref{probform2d} can be equivalently written as a fixed--point problem: 
 Find $\bu_h \in \mathcal{O}_h$ (where the discrete space is defined in \eqref{eq:Oh}) such that 
\begin{equation*}
\mathcal{J}_h(\bu_h)=\bu_h,
\end{equation*}
{where $\mathcal{J}_h: \mathcal{O}_h \subset V_h \to V_h$ is defined by $\bbeta_h \mapsto \mathcal{J}(\bbeta_h)=\bu_h$, where $(\bu_h,\bomega_h,p_h) \in V_h \times W_h\times Q_h$ is the unique solution of the discrete Oseen problem:} 
 Find $(\bu_h,\bomega_h,p_h)\in V_h\times W_h\times Q_h$ such that 
\begin{align*}
a((\bu_h,\bomega_h),(\bv_h,\btheta_h))+N(\bbeta_h;\bu_h, \bv_h)+b(\bv_h,p_h)=&\;G(\bv_h,
\btheta_h) 
& \forall(\bv_h,\btheta_h)\in V_h\times W_h,\\
b(\bu_h,q_h)=&\;0 &\forall q_h\in Q_h.
\end{align*}

{In order to show the unique solvability of  \eqref{probform2d}, we proceed to 
first establish the well-definedness of the operator $\mathcal{J}_h$, then that $\mathcal{J}_h(\mathcal{O}_h)\subset \mathcal{O}_h$,  and then that $\mathcal{J}_h$ is a contraction.} 

\begin{lemma}\label{punto_fijo_d}
{Under the same assumptions as in Lemma~\ref{punto_fijo}} for $\kappa_1>0$ and as in Lemma~\ref{punto_fijo2} for $\kappa_1=0$;  
we choose $0<\delta<\frac{{\overline{\alpha}(\nu)} }{C_0^2C_4^2 d^{1/2}}$ with $\overline{\alpha}(\nu)$ as defined in \eqref{alpha1}, \eqref{alpha2} and denote the set
\begin{equation}\label{eq:Oh}
\mathcal{O}_h:=\{{\bu_h}\in V_h: |{\bu_h}|_{1,\O}\leq \delta\}.
\end{equation}
If
$
\Vert \ff \Vert_{0,\Omega} <\frac{1}{2}{ \overline{\alpha}(\nu)}  \delta,
$
then   the operator $\mathcal{J}_h$ has a unique fixed point in $\mathcal{O}_h$.
\end{lemma}
\begin{proof}
Provided that we can choose $\btheta_h=\curl \bu_h$ in the steps \eqref{curl1}, \eqref{curl2}, the proof is {analogous to that of Lemma} \ref{punto_fijo} in the $\kappa_1>0$ case and Lemma \ref{punto_fijo2} in the case $\kappa_1=0$. This issue motivates the need of taking $W_h$ such that $\curl V_h \subset W_h$.
\end{proof}

And as an immediate consequence of Lemma~\ref{punto_fijo_d}, we have the following result.
\begin{theorem}
Under the assumptions of Lemma~\ref{punto_fijo_d}, the operator $\mathcal{J}_h$ has a unique fixed point $\bu_h \in \mathcal{O}_h$. Equivalently, problem \eqref{probform2d} has a unique solution $(\bu_h,\bomega_h,p_h) \in V_h\times W_h\times Q_h$ satisfying  $\bu_h \in \mathcal{O}_h.$
\end{theorem} 

\subsection{The C\'ea estimate}
Our next objective is to obtain  a best approximation result for \eqref{probform2d}. Let $(\bu,\bomega,p)$ and $(\bu_h,\bomega_h,p_h)$ be the unique solutions of \eqref{probform2} and \eqref{probform2d}, respectively. It is readily observed that the following error equation is satisfied:
\begin{align}\label{error-eq}
\nonumber 
a((e_{\bu},e_{\bomega}),(\bv_h,\btheta_h))+N(\bu;\bu, \bv_h) -N(\bu_h;\bu_h, \bv_h)+b(\bv_h,e_{p})&=\;0&  \forall(\bv_h,\btheta_h)\in V_h\times W_h,\\
  b(e_{\bu},q_h)& =\;0 &\forall q_h\in Q_h,
\end{align}
where $e_{\bu}:=\bu-\bu_h$,  $e_{\bomega}:=\bomega-\bomega_h$,  $e_{p}:=p-p_h$, denote  the corresponding errors. 
Given $(\bv_h,\btheta_h,q_h)\in V_h\times W_h\times Q_h$, {we decompose these errors as} 
\begin{gather*}
e_{\bu}=\bz_{\bu}+\bx_{\bu}=(\bu-\bv_h)+(\bv_h-\bu_h),\quad 
 e_{\bomega}=\bz_{\bomega}+\bx_{\bomega}=(\bomega-\btheta_h)+(\btheta_h-\bomega_h),\\
e_{p}=z_p+x_p=(p-q_h)+(q_h-p_h),\end{gather*}
where $\bx_{\bu} \in V_h, \bx_{\bomega} \in W_h$ and $x_p \in Q_h$.
With these notations in hand, we have the following result, valid for $\kappa_1>0$.

\begin{theorem}\label{Cea} 
Consider the  assumptions of Lemma~\ref{punto_fijo} and 
choose $0<\delta<\frac{{\overline{\alpha}} }{2C_4^2 d^{1/2}}$.
If 
$
\Vert \ff \Vert_{0,\Omega} <\frac{1}{2}{\overline{\alpha}}  \delta,
$
then there exists a positive constant ${C}$ independent of $h$ such that  
\[
\Vert (\bu-\bu_h,\bomega-\bomega_h)\Vert_{X} + \| p-p_h\|_{0,\Omega} \leq   {C}\displaystyle\inf_{(\bv_h,\btheta_h,q_h)\in V_h\times W_h\times Q_h}  \{\ric{\Vert (\bu-\bv_h,\bomega-\btheta_h)\Vert_{X}}
+\Vert p-q_h\Vert_{0,\O} \}{.}
\]
\end{theorem}
\begin{proof}
Let $(\bv_h,\btheta_h,q_h)\in V_h\times W_h\times Q_h$. From the definition of $a$ and grouping terms, we have 
\begin{align*}
a((\bx_{\bu},\bx_{\bomega}),(\bx_{\bu},\bx_{\bomega}))&=a((e_{\bu},e_{\bomega}),(\bx_{\bu},\bx_{\bomega}))-(\sigma \bz_{\bu},\bx_{\bu})_{0,\O}-(\nu \bz_{\bomega}, \bx_{\bomega})_{0,\O}-(\nu \bz_{\bomega}, \curl \bx_{\bu})_{0,\O}\\
&\quad+(\nu \bx_{\bomega}, \curl \bz_{\bu})_{0,\O} -\kappa_{1}(\curl \bz_{\bu}, \curl \bx_{\bu})_{0,\O}-\kappa_2(\vdiv \bz_{\bu},\vdiv \bx_{\bu})_{0,\O}\\
&\quad+\kappa_1(\bz_{\bomega},\curl \bx_{\bu})_{0,\O}+2(\beps(\bz_{\bu})\nabla \nu, \bx_{\bu})_{0,\O}-(\bz_{\bomega},\nabla \nu \times \bx_{\bu})_{0,\O}.
\end{align*}
Next we invoke  the ellipticity of $a(\cdot,\cdot)$, from which it follows that
\begin{align*}
{\overline{\alpha}}(\nu) \Vert (\bx_{\bu},\bx_{\bomega})\Vert_{{X}}^2 &\leq a((e_{\bu},e_{\bomega}),(\bx_{\bu},\bx_{\bomega}))+\sigma \Vert \bz_{\bu} \Vert_{0,\O}\Vert \bx_{\bu}\Vert_{0,\O}+\nu_1\Vert \bz_{\bomega}\Vert_{0,\O} \Vert \bx_{\bomega}\Vert_{0,\O}   +\nu_1\Vert \bz_{\bomega}\Vert_{0,\O}\Vert \curl \bx_{\bu} \Vert_{0,\O}\\
&\quad+\nu_1\Vert \bx_{\bomega}\Vert_{0,\O}\Vert \curl \bz_{\bu} \Vert_{0,\O}+\kappa_1\Vert \curl \bz_{\bu}\Vert_{0,\O}\Vert \curl \bx_{\bu}\Vert_{0,\O} +\kappa_2\Vert \vdiv \bz_{\bu}\Vert_{0,\O}\Vert \vdiv \bx_{\bu}\Vert_{0,\O}\\
&\quad+\kappa_1 \Vert \bz_{\bomega}\Vert_{0,\O}\Vert \curl \bx_{\bu}\Vert_{0,\O}+2C_r d^{\frac{r-2}{2r}}\Vert \nabla \nu \Vert_{0,r^*,\O}\Vert \nabla \bz_{\bu}\Vert_{0,\O}\Vert \bx_{\bu}\Vert_{1,\O} \\
&\quad+ 2C_r d^{\frac{r-2}{2r}}\Vert \nabla \nu \Vert_{0,r^*,\O}\Vert  \bz_{\bomega}\Vert_{0,\O}\Vert \bx_{\bu}\Vert_{1,\O},
\end{align*} 
{and} arranging terms, we obtain
\begin{equation*}
\Vert (\bx_{\bu},\bx_{\bomega})\Vert_{{X}}^2 \leq {\overline{\alpha}}(\nu)^{-1}a((e_{\bu},e_{\bomega}),(\bx_{\bu},\bx_{\bomega})) +(m_1\Vert \bz_{\bomega}\Vert_{0,\O}+m_2|\bz_{\bu}|_{1,\Omega})\Vert (\bx_{\bu},\bx_{\bomega})\Vert_X,
\end{equation*}
where $m_1$ and $m_2$ are given by 
\begin{align*}
m_1&:={\overline{\alpha}}(\nu)^{-1}\max\{2\nu_1,\kappa_1,{2C_0C_r d^{\frac{r-2}{2r}}\Vert \nabla \nu \Vert_{0,r^*,\O}}\}, \\ m_2&:={\overline{\alpha}}(\nu)^{-1}\max\{\sigma_1 C_0^2,\nu_1,\kappa_1,\kappa_2,{2C_0C_r d^{\frac{r-2}{2r}}\Vert \nabla \nu \Vert_{0,r^*,\O}}\}.
\end{align*}

On the other hand, from \eqref{error-eq} we have
\begin{align*}
a((e_{\bu},e_{\bomega}),(\bx_{\bu},\bx_{\bomega}))&=-N(\bu;\bu, \bx_{\bu})+N(\bu_h;\bu_h, \bx_{\bu})-{b}((\bx_{\bu},\bx_{\bomega}),e_{p})\\
&=-N(\bu;\bz_{\bu}, \bx_{\bu})-N(\bu;\bx_{\bu}, \bx_{\bu})-N(\bz_{\bu};\bu_h, \bx_{\bu})\\
&\qquad -N(\bx_{\bu};\bu_h, \bx_{\bu})+( z_p,\vdiv \bx_{\bu})_{0,\O}+( x_p,\vdiv \bx_{\bu})_{0,\O}\\
&\leq (2C_0^2C_4^2d^{1/2} \delta| \bz_{\bu}|_{1,\O}+\Vert z_p\Vert_{0,\O})\Vert (\bx_{\bu},\bx_{\bomega})\Vert_X\\
& \qquad \qquad+2C_0^2C_4^2d^{1/2}\delta\Vert (\bx_{\bu},\bx_{\bomega})\Vert_X^2 -(x_p,\vdiv \bv_h)_{0,\O},
\end{align*}
where we have used \ric{Cauchy--Schwarz} inequality. 
From the above inequality, we obtain
\begin{align*}
\biggl( 1-\frac{2C_0^2C_4^2d^{1/2}\delta}{{\overline{\alpha}}(\nu) }\biggr)\Vert (\bx_{\bu},\bx_{\bomega})\Vert_{{X}} &\leq \left(m_2+\frac{2C_0^2C_4^2d^{1/2}\delta}{{\overline{\alpha}}(\nu) }\right)| \bz_{\bu}|_{1,\Omega}+m_1\Vert \bz_{\bomega}\Vert_{0,\O} 
\\
&\qquad+{\overline{\alpha}}(\nu) ^{-1}(\Vert z_p\Vert_{0,\O}-(x_p,\vdiv \bv_h)_{0,\O}),
\end{align*}
while by the hypothesis of our problem is easy to see that $ 1-\frac{2C_0^2C_4^2d^{1/2}\delta}{{\overline{\alpha}}(\nu) }>0.$
Therefore,
\begin{equation}\label{cotau}
\Vert (\bx_{\bu},\bx_{\bomega})\Vert_{ {X}} \leq M_2 \Vert \bz_{\bu}\Vert_{1,\Omega}+M_1\Vert \bz_{\bomega}\Vert_{0,\O} + M_3\Vert z_p\Vert_{0,\O}-M_3(x_p,\vdiv \bv_h)_{0,\O},
\end{equation}
where the constants are given by 
\begin{align*}
M_0&:=\left( 1-\frac{2C_0^2C_4^2d^{1/2}\delta}{{\overline{\alpha}}(\nu) }\right)^{-1}, \text{ } M_2:=M_0\left( m_2+\frac{2C_0^2C_4^2d^{1/2}\delta}{{\overline{\alpha}}(\nu) }\right),\\
M_1&:=m_1M_0,\text{ } M_3:={\overline{\alpha}}(\nu)^{-1}M_0
\end{align*}

Next, after introducing the following subspace of $V_h:$ 
$$Z_h:=\left \{ \bv_h\in V_h: \quad  (q_h,\vdiv \bv_h)_{0,\O}=0, \qquad \forall q_h\in Q_h \right \},$$
we can deduce from \eqref{cotau} that 
\begin{equation*}
\Vert (\bu_h-\boldsymbol{z}_h,\bx_{\bomega})\Vert_{{X}} \leq M_2 | \bu-\boldsymbol{z}_h|_{1,\Omega}+M_1\Vert \bz_{\bomega}\Vert_{0,\O} + M_3\Vert z_p\Vert_{0,\O} \qquad \forall  \boldsymbol{z}_h\in Z_h.
\end{equation*}

As $V_h$ and $Q_h$ are stable for the Stokes problem, then, using
\cite[Proposition 7.4.1]{Q}  we have that $V_h=Z_h\oplus (Z_h)^\bot$ and there exists $\gamma_1>0$ such that
\begin{equation}\label{gamo2}
\gamma_1 \Vert \boldsymbol{o}_h\Vert_{1,\O}\leq \sup_{q_h\in Q_h}\frac{(q_h, \vdiv \boldsymbol{o}_h)_{0,\O}}{\Vert q_h\Vert_{0,\O}} \qquad \forall  \boldsymbol{o}_h\in (Z_h)^\bot.
\end{equation}
From the above, we deduce that for $\bv_h$, there exist $\boldsymbol{z}_h\in Z_h$ and $\boldsymbol{o}_h\in (Z_h)^\bot$ such that $\bv_h=\boldsymbol{z}_h+\boldsymbol{o}_h.$                                                                                                                                                                                                                                                                                                                                                                                                                                                                                                                                                                                                                                                                                                                                                                                                                                                                                                                                                                                                                                                                                                                                                                                                                                                                                                                                                                                                                                                                                                                                                                                                                                                                                                                                                                                                                                                                                                                                                                                                                                                                                                           
By the triangle inequality, we find 
\begin{align*}
| \bu_h-\bv_h|_{1,\O} &\leq | \bu_h-\boldsymbol{z}_h|_{1,\O}+| \boldsymbol{o}_h|_{1,\O} \\
&\leq | \bu_h-\boldsymbol{z}_h|_{1,\O}+\gamma_1^{-1}\sup_{q_h\in Q_h}\frac{(q_h, \vdiv(-\bu+\boldsymbol{z}_h+ \boldsymbol{o}_h))_{0,\O}}{\Vert q_h\Vert_{0,\O}}\\
&\leq | \bu_h-\boldsymbol{z}_h|_{1,\O}+ \gamma_1^{-1} \Vert \vdiv(-\bu+\bv_h)\Vert_{0,\O}\\
&\leq | \bu_h-\boldsymbol{z}_h|_{1,\O}+ \gamma_1^{-1} | \bu-\bv_h|_{1,\O},
\end{align*}
and so 
$$
| \bx_{\bu}|_{1,\O}\leq | \bu_h-\boldsymbol{z}_h|_{1,\O}+ \gamma_1^{-1} | \bz_{\bu}|_{1,\O}.
$$
Then, after algebraic manipulations and using \eqref{gamo2}, we are left with 
\begin{align*}
\Vert (\bx_{\bu},\bx_{\bomega})\Vert_{{X}} &\leq | \bx_{\bu}|_{1,\O}+ \Vert \bx_{\bomega}\Vert_{0,\O}\leq | \bu_h-\boldsymbol{z}_h|_{1,\O}+ \gamma_1^{-1} | \bz_{\bu}|_{1,\O}+ \Vert \bx_{\bomega}\Vert_{0,\O}\\
&\leq \sqrt{2}(| \bu_h-\boldsymbol{z}_h|_{1,\O}^2+\Vert \bx_{\bomega}\Vert_{0,\O}^2)^{1/2}+\gamma_1^{-1} | \bz_{\bu}|_{1,\O}\\
&=\sqrt{2}\Vert (\bu_h-\boldsymbol{z}_h,\bx_{\bomega})\Vert_{{X}}+\gamma_1^{-1} | \bz_{\bu}|_{1,\O}\\
&\leq \sqrt{2}(M_2( | \bu-\boldsymbol{v}_h|_{1,\Omega}+| \boldsymbol{o}_h|_{1,\Omega})+M_1\Vert \bz_{\bomega}\Vert_{0,\O} + M_3\Vert z_p\Vert_{0,\O})+\gamma_1^{-1} | \bz_{\bu}|_{1,\O},\\
&\leq \sqrt{2}(M_2( | \bz_{\bu}|_{1,\Omega}+\gamma_1^{-1}| \bz_{\bu}|_{1,\Omega})+M_1\Vert \bz_{\bomega}\Vert_{0,\O} + M_3\Vert z_p\Vert_{0,\O})+\gamma_1^{-1} | \bz_{\bu}|_{1,\O},\\
& \leq  (\sqrt{2}M_2 +\gamma_1^{-1}(1+\sqrt{2} ))| \bz_{\bu}|_{1,\Omega}+\sqrt{2}M_1\Vert \bz_{\bomega}\Vert_{0,\O} +\sqrt{2} M_3\Vert z_p\Vert_{0,\O}.
\end{align*}

On the other hand, from \eqref{inf-sup-d} and \eqref{error-eq}, we can assert \ric{that
\begin{align*}
\Vert x_p\Vert_{0,\O} & \leq \gamma_0^{-1}\sup_{\bv_h\in V_h}\frac{b(\bv_h,x_p)}{|\bv_h|_{1,\Omega}}\\
&=\gamma_0^{-1}\sup_{\bv_h\in V_h}\frac{-a((e_{\bu},e_{\bomega}),(\bv_h,\btheta_h))-N(\bu;\bu, \bv_h)+N(\bu_h;\bu_h, \bv_h)-b(\bv_h,z_p)}{|\bv_h|_{1,\Omega}}\\
&\leq \gamma_0^{-1}(\Vert a\Vert +2C_0^2C_{4}^2d^{1/2}\delta)\Vert (e_{\bu},e_{\bomega})\Vert_{{X}}+\gamma_0^{-1}\Vert z_p\Vert_{0,\O}.
\end{align*}
Therefore}, we end up with the bounds 
\begin{align*}
 \Vert (e_{\bu},e_{\bomega},e_{p})\Vert_{{H}} 
& \leq \Vert (\bz_{\bu},\bz_{\bomega},z_p)\Vert_{{H}}+\Vert (\bx_{\bu},\bx_{\bomega},x_p)\Vert_{{H}} \leq \Vert (\bz_{\bu},\bz_{\bomega},z_p)\Vert_{{H}}+\Vert (\bx_{\bu},\bx_{\bomega})\Vert_{{X}}+\Vert x_p\Vert_{0,\O}\\
& \leq \Vert (\bz_{\bu},\bz_{\bomega},z_p)\Vert_{{H}}+\widehat{M}_2 | \bz_{\bu}|_{1,\Omega}+\widehat{M}_1\Vert \bz_{\bomega}\Vert_{0,\O} + \widehat{M}_3\Vert z_p\Vert_{0,\O}\\
&\quad +\gamma_0^{-1}(\Vert a\Vert +2C_0^2C_{4}^2d^{1/2}\delta)\Vert (e_{\bu},e_{\bomega})\Vert_{{X}}+\gamma_0^{-1}\Vert z_p\Vert_{0,\O}\\
&\leq \Vert (\bz_{\bu},\bz_{\bomega},z_p)\Vert_{{H}}+\widehat{M}_2 | \bz_{\bu}|_{1,\Omega}+\widehat{M}_1\Vert \bz_{\bomega}\Vert_{0,\O} + \widehat{M}_3\Vert z_p\Vert_{0,\O}+\gamma_0^{-1}\Vert z_p\Vert_{0,\O}\\
&\quad+\widehat{M}_4\Vert (\bz_{\bu},\bz_{\bomega})\Vert_{{X}}+\widehat{M}_4 \widehat{M}_2 | \bz_{\bu}|_{1,\Omega}+\widehat{M}_4 \widehat{M}_1\Vert \bz_{\bomega}\Vert_{0,\O} +\widehat{M}_4 \widehat{M}_3\Vert z_p\Vert_{0,\O}.
\end{align*}
And consequently, we obtain the estimate 
\begin{equation*}
\Vert (e_{\bu},e_{\bomega},e_{p})\Vert_{{H}}\leq (1+\widehat{M}_4)(1+\widehat{M}_1+\widehat{M}_2)\Vert (\bz_{\bu},\bz_{\bomega})\Vert_{{X}}+\widehat{M}_5\Vert z_p\Vert_{0,\O},
\end{equation*}
where the involved constants are 
\begin{gather*}
\widehat{M}_1=\sqrt{2}M_1, \quad \widehat{M}_2=\sqrt{2}M_2+\gamma_1^{-1}(1+\sqrt{2}), \quad \widehat{M}_3=\sqrt{2}M_3,\\
 \widehat{M}_4=\gamma_0^{-1}(\Vert a\Vert +2C_{4}^2d^{1/2}\delta), \quad \widehat{M}_5=1+\gamma_0^{-1}+(1+\widehat{M}_4)\widehat{M}_3.
 \end{gather*} 
\end{proof}

\ruben{
For the case $\kappa_1=0$ the results reads as follows.
\begin{theorem}\label{Cea1} 
Consider the  assumptions of Lemma~\ref{punto_fijo2} and 
choose $0<\delta<\frac{{\overline{\alpha}(\nu)} }{2C_0^2C_4^2 d^{1/2}}$.
If 
$
\Vert \ff \Vert_{0,\Omega} <\frac{1}{2}{\overline{\alpha}(\nu)}  \delta,
$
then there exists a positive constant ${C}$ independent of $h$ such that  
\[
\Vert (\bu-\bu_h,\bomega-\bomega_h)\Vert_{X} + \| p-p_h\|_{0,\Omega} \leq   {C}\displaystyle\inf_{(\bv_h,\btheta_h,q_h)\in V_h\times W_h\times Q_h}  \{\ric{\Vert (\bu-\bv_h,\bomega-\btheta_h)\Vert_{X}}
+\Vert p-q_h\Vert_{0,\O} \}{.}
\]
\end{theorem}
\begin{proof}
Let $(\bv_h,\btheta_h,q_h)\in V_h\times W_h\times Q_h$. From the definition of $a$ and grouping terms, we have 
\begin{align*}
a((\bx_{\bu},\bx_{\bomega}),(\bx_{\bu},\bx_{\bomega}))&=a((e_{\bu},e_{\bomega}),(\bx_{\bu},\bx_{\bomega}))-(\sigma \bz_{\bu},\bx_{\bu})_{0,\O}-(\nu \bz_{\bomega}, \bx_{\bomega})_{0,\O}\\
&\quad -(\nu \bz_{\bomega}, \curl \bx_{\bu})_{0,\O}+(\nu \bx_{\bomega}, \curl \bz_{\bu})_{0,\O}-\kappa_2(\vdiv \bz_{\bu},\vdiv \bx_{\bu})_{0,\O}\\
& \qquad \quad+2(\beps(\bz_{\bu})\nabla \nu, \bx_{\bu})_{0,\O}-(\bz_{\bomega},\nabla \nu \times \bx_{\bu})_{0,\O},
\end{align*}
and from \eqref{error-eq} we have
$$
a((\bx_{\bu},\bx_{\bomega}),(\boldsymbol{0},\btheta_{h}))=-a((\bz_{\bu},\bz_{\bomega}),(\boldsymbol{0},\btheta_{h}))=-(\nu \bz_{\bomega},\btheta_h)_{0,\Omega}+(\nu \curl \bz_{\bu},\btheta_h)_{0,\Omega}, \text{ for all } \btheta_h \in W_h.
$$
On the other hand 
\begin{align*}
a((\bx_{\bu},\bx_{\bomega}),(\bx_{\bu},\bx_{\bomega}-c\curl \bx_{\bu}))&=a((\bx_{\bu},\bx_{\bomega}),(\bx_{\bu},\bx_{\bomega}))+a((\bx_{\bu},\bx_{\bomega}),(\boldsymbol{0},-c\curl \bx_{\bu}))\\
&=a((\bx_{\bu},\bx_{\bomega}),(\bx_{\bu},\bx_{\bomega}))-a((\bz_{\bu},\bz_{\bomega}),(\boldsymbol{0},-c\curl \bx_{\bu})).
\end{align*}
Then taking $\varepsilon_2 \in (0,2)$ and $c=\dfrac{\nu_0^2}{\nu_1^2}\left( 1-\dfrac{\varepsilon_2}{2}\right)$ in \eqref{bounds00}, we have that
\begin{align*}
{\overline{\alpha}}(\nu) \Vert (\bx_{\bu},\bx_{\bomega})\Vert_{{X}}^2 &\leq a((e_{\bu},e_{\bomega}),(\bx_{\bu},\bx_{\bomega}))+\sigma_1 C_0^2 | \bz_{\bu} |_{1,\O}| \bx_{\bu}|_{1,\O}\\
& \qquad +\nu_1\Vert \bz_{\bomega}\Vert_{0,\O} \Vert \bx_{\bomega}\Vert_{0,\O}   +\nu_1\Vert \bz_{\bomega}\Vert_{0,\O}\Vert \curl \bx_{\bu} \Vert_{0,\O}\\
&\leq a((e_{\bu},e_{\bomega}),(\bx_{\bu},\bx_{\bomega}))+ K(\sigma,\nu,\kappa)\Vert (\bz_{\bu},\bz_{\bomega})\Vert_{{X}} \Vert (\bx_{\bu},\bx_{\bomega})\Vert_{{X}},
\end{align*}
where $K(\sigma,\nu,\kappa)=K_0(\sigma,\nu,\kappa)K_1(\sigma,\nu,\kappa)$ with
\begin{align*}
K_0(\sigma,\nu,\kappa)&=2^{1/2} \max\left\{\sigma_1C_0^2+\nu_1(2+c)+2C_r C_0 d^{\frac{r-2}{2r}} \Vert \nabla \nu \Vert_{0,r^{*},\Omega},\kappa_2\right\}^{1/2}, \\
K_1(\sigma,\nu,\kappa)&=2^{1/2}\max\left\{\sigma_1C_0^2+\nu_1(1+2c)+4C_r C_0 d^{\frac{r-2}{2r}} \Vert \nabla \nu \Vert_{0,r^{*},\Omega},\kappa_2\right\}^{1/2}.
\end{align*}
On the other hand, from \eqref{error-eq} we have
\begin{align*}
a((e_{\bu},e_{\bomega}),(\bx_{\bu},\bx_{\bomega}))&=-N(\bu;\bu, \bx_{\bu})+N(\bu_h;\bu_h, \bx_{\bu})-{b}((\bx_{\bu},\bx_{\bomega}),e_{p})\\
&=-N(\bu;\bz_{\bu}, \bx_{\bu})-N(\bu;\bx_{\bu}, \bx_{\bu})-N(\bz_{\bu};\bu_h, \bx_{\bu})\\
&\qquad \qquad -N(\bx_{\bu};\bu_h, \bx_{\bu})+( z_p,\vdiv \bx_{\bu})_{0,\O}+( x_p,\vdiv \bx_{\bu})_{0,\O}\\
&\leq (2C_0^2C_4^2d^{1/2} \delta | \bz_{\bu}|_{1,\O}+\Vert z_p\Vert_{0,\O})\Vert (\bx_{\bu},\bx_{\bomega})\Vert_X\\
& \qquad+2C_0^2C_4^2d^{1/2}\delta\Vert (\bx_{\bu},\bx_{\bomega})\Vert_X^2 -(x_p,\vdiv \bv_h)_{0,\O},
\end{align*}
where we have used Cauchy–Schwarz inequality.

From the above inequality, we obtain
\begin{align*}
\biggl( 1-\frac{2C_0^2C_4^2d^{1/2}\delta}{{\overline{\alpha}(\nu)} }\biggr)\Vert (\bx_{\bu},\bx_{\bomega})\Vert_{{X}}^2 & \leq \overline{\alpha}(\nu)^{-1}(K(\sigma,\nu,\kappa_2)+2C_0^2C_4^2d^{1/2})\Vert (\bz_{\bu},\bz_{\bomega})\Vert_X\Vert (\bx_{\bu},\bx_{\bomega})\Vert_X\\
&\quad +{\overline{\alpha}}(\nu)^{-1}\Vert z_p\Vert_{0,\O}\Vert (\bx_{\bu},\bx_{\bomega})\Vert_X-{\overline{\alpha}}(\nu)^{-1}(x_p,\vdiv \bv_h)_{0,\O},
\end{align*}
while, from de assumtion on $\delta$, it is easy to see that $ 1-\frac{2C_0^2C_4^2d^{1/2}\delta}{{\overline{\alpha}(\nu)} }>0.$ Therefore,
\begin{equation}\label{cotau}
\Vert (\bx_{\bu},\bx_{\bomega})\Vert_{ {X}}^2 \leq M_1(\sigma,\nu,\kappa_2) \Vert (\bz_{\bu},\bz_{\bomega})\Vert_X\Vert (\bx_{\bu},\bx_{\bomega})\Vert_X+M_2(\nu)\Vert z_p\Vert_{0,\O}\Vert (\bx_{\bu},\bx_{\bomega})\Vert_X-M_2(\nu)(x_p,\vdiv \bv_h)_{0,\O},
\end{equation}
where the constants are given by 
$$M_1(\sigma,\nu,\kappa_2)=\dfrac{1}{\overline{\alpha}(\nu)}\left( 1-\frac{2C_0^2C_4^2d^{1/2}\delta}{{\overline{\alpha}(\nu)} }\right)^{-1}(K(\sigma,\nu,\kappa_2)+2C_0^2C_4^2d^{1/2}), \text{ } M_2(\nu)=\dfrac{1}{\overline{\alpha}(\nu)}\left( 1-\frac{2C_0^2C_4^2d^{1/2}\delta}{{\overline{\alpha}} }\right)^{-1}.$$
we can deduce from \eqref{cotau} that 
\begin{equation*}
\Vert (\bu_h-\boldsymbol{z}_h,\bx_{\bomega})\Vert_{{X}} \leq M_1(\sigma,\nu,\kappa_2)\Vert (\bu-\boldsymbol{z}_h,\bz_{\bomega})\Vert_X\Vert (\bx_{\bu},\bx_{\bomega})\Vert_X+M_2(\nu)\Vert z_p\Vert_{0,\O}\Vert (\bx_{\bu},\bx_{\bomega})\Vert_X \qquad \forall  \boldsymbol{z}_h\in Z_h.
\end{equation*}

Then, after algebraic manipulations and using \eqref{gamo2}, {we have}
\begin{align*}
\Vert (\bx_{\bu},\bx_{\bomega})\Vert_{{X}} &\leq | \bx_{\bu}|_{1,\O}+ \Vert \bx_{\bomega}\Vert_{0,\O}\leq | \bu_h-\boldsymbol{z}_h|_{1,\O}+ \gamma_1^{-1} | \bz_{\bu}|_{1,\O}+ \Vert \bx_{\bomega}\Vert_{0,\O}\\
&\leq \sqrt{2}(| \bu_h-\boldsymbol{z}_h|_{1,\O}^2+\Vert \bx_{\bomega}\Vert_{0,\O}^2)^{1/2}+\gamma_1^{-1} | \bz_{\bu}|_{1,\O}=\sqrt{2}\Vert (\bu_h-\boldsymbol{z}_h,\bx_{\bomega})\Vert_{{X}}+\gamma_1^{-1} | \bz_{\bu}|_{1,\O}\\
&\leq \sqrt{2}(M_1(\sigma,\nu,\kappa_2)( | \bu-\boldsymbol{v}_h|_{1,\Omega}+| \boldsymbol{o}_h|_{1,\Omega})+M_1(\sigma,\nu,\kappa_2)\Vert \bz_{\bomega}\Vert_{0,\O} + M_2(\nu)\Vert z_p\Vert_{0,\O})+\gamma_1^{-1} | \bz_{\bu}|_{1,\O}\\
& \leq  (\sqrt{2}M_1(\sigma,\nu,\kappa_2)+\gamma_1^{-1}(1+\sqrt{2} ))| \bz_{\bu}|_{1,\Omega}+\sqrt{2}M_1(\sigma,\nu,\kappa_2)\Vert \bz_{\bomega}\Vert_{0,\O} +\sqrt{2} M_2(\nu)\Vert z_p\Vert_{0,\O}.
\end{align*}
On the other hand, from \eqref{inf-sup-d} and \eqref{error-eq}, we can assert that 
\begin{align*}
\Vert x_p\Vert_{0,\O} & \leq \gamma_0^{-1}\sup_{\bv_h\in V_h}\frac{b(\bv_h,x_p)}{|\bv_h|_{1,\Omega}}\\
&=\gamma_0^{-1}\sup_{\bv_h\in V_h}\frac{-a((e_{\bu},e_{\bomega}),(\bv_h,\btheta_h))-N(\bu;\bu, \bv_h)+N(\bu_h;\bu_h, \bv_h)-b(\bv_h,z_p)}{|\bv_h|_{1,\Omega}}\\
&\leq \gamma_0^{-1}(\Vert a\Vert +2C_0^2C_{4}^2d^{1/2}\delta)\Vert (e_{\bu},e_{\bomega})\Vert_{{X}}+\gamma_0^{-1}\Vert z_p\Vert_{0,\O}.
\end{align*}
Finally we get the estimates
\begin{align*}
\Vert (e_{\bu},e_{\bomega})\Vert_{X}&\leq (1+\sqrt{2}M_1(\sigma,\nu,\kappa_2) +\gamma_1^{-1}(1+\sqrt{2} ))| \bz_{\bu}|_{1,\Omega}+(1+\sqrt{2}M_1(\sigma,\nu,\kappa_2))\Vert \bz_{\bomega}\Vert_{0,\O} +\sqrt{2} M_2(\nu)\Vert z_p\Vert_{0,\O},\\
\Vert x_p\Vert_{0,\O} &\leq \gamma_0^{-1}(\Vert a\Vert +2C_0^2C_{4}^2d^{1/2}\delta)\Vert (e_{\bu},e_{\bomega})\Vert_{{X}}+\gamma_0^{-1}\Vert z_p\Vert_{0,\O}.
\end{align*}
Then the constant in the C\'ea estimate depends also on the augmentation parameter $\kappa_2$ in the following manner
$$
\Vert (e_{\bu},e_{\bomega})\Vert_{X}+\Vert e_p\Vert_{0,\O}\leq C(\sigma,\nu,\kappa_2)(\Vert (\bz_{\bu},\bz_{\bomega})\Vert_{X}+\Vert \bz_p\Vert_{0,\O}),
$$
where
\begin{align*}
C(\sigma,\nu,\kappa_2)&=\max\{D_1(1+D_0),D_2(1+D_0),1+\sqrt{2}M_2(1+D_0)+\gamma_0^{-1}\}, \text{ }D_0=\gamma_0^{-1}(\Vert a\Vert +2C_0^2C_{4}^2d^{1/2}\delta),\\
D_1&=1+\sqrt{2}M_1 +\gamma_1^{-1}(1+\sqrt{2} ),\text{ }D_2=1+\sqrt{2}M_1.
\end{align*}
\end{proof} 
}

\subsection{Discrete subspaces and error estimates}\label{sec:FE}
In this section we give examples of finite element spaces for $V_h,Q_h,W_h$ 
and derive the corresponding rate of convergence
for each finite element family.

\noindent\textbf{Generalised Taylor--Hood--$\mathbb{P}_{k}$.} 
We begin with  Taylor--Hood finite
elements \cite{HT} to approximate  velocity and pressure, and we will contemplate continuous or discontinuous piecewise polynomial spaces for vorticity. 
More precisely, for any $k\ge1$, we consider
\begin{equation}\label{set1}
\begin{split}
V_h:&=\{\bv_h \in C(\overline{\Omega})^d:\bv_h\vert_{K}\in \mathbb{P}_{k+1}(K)^d \quad \forall K \in \mathcal{T}_h\} \cap \HCUO^d,\\
Q_h:&=\{ q_h \in C(\overline{\Omega}): q_h\vert_{K}\in \mathbb{P}_{k}(K) \quad \forall K \in \mathcal{T}_h\} \cap \LOO,\\
W_h^1:&=\{\btheta_h \in C(\overline{\Omega})^{d(d-1)/2}: \btheta_h\vert_{K}\in \mathbb{P}_{k}(K)^{d(d-1)/2} \quad \forall K \in \mathcal{T}_h \},\\
W_h^2:&=\{\btheta_h \in \LO^{d(d-1)/2}: \btheta_h\vert_{K}\in\mathbb{P}_{k}(K)^{d(d-1)/2} \quad \forall K \in \mathcal{T}_h \}.
\end{split}
\end{equation}

It is well known that $(V_h,Q_h)$ satisfies the inf-sup condition
\eqref{inf-sup-d} (see \cite{gr-1986}). Let us recall   approximation properties of the
finite element subspaces  \eqref{set1}, \ric{obtained using the classical Lagrange interpolator}.  
Assume that $\bu \in H^{1+s}(\O)^d$, $p \in H^{s}(\O)$
and $\bomega \in H^{s}(\O)^{d(d-1)/2}$, for some
$s\in(1/2,k+1]$. Then there exists $C>0$,
independent of $h$, such that
\begin{subequations}
\begin{align}
\inf_{\bv_h \in V_h}\Vert\bu -\bv_h\Vert_{1,\O} &\leq C h^{s}\Vert \bu\Vert_{1+s,\O},\label{Ap1}\qquad 
\inf_{q_h \in Q_h}\Vert p -q_h \Vert_{0,\O} \leq C h^{s}\Vert p \Vert_{s,\O},
\\
\inf_{\btheta_h \in W_h^1}\Vert \bomega -\btheta_h \Vert_{0,\O}
&\leq C h^{s}\Vert \bomega\Vert_{s,\O},\qquad 
\inf_{\btheta_h \in W_h^2}\Vert \bomega -\btheta_h \Vert_{0,\O}
\leq C h^{s}\Vert \bomega\Vert_{s,\O}.\label{Ap33}
\end{align}
\end{subequations}

The following theorem provides the rate of convergence expected when using 
  \eqref{probform2d}. The proof follows directly from Theorem \ref{Cea} in combination with  \eqref{Ap1}-\eqref{Ap33}.
\begin{theorem}\label{teoTaylorHood}
Let $k\ge1$ be a integer and let $V_h,Q_h$ and $W^i_h$, $i=1,2$
be given by \eqref{set1}.
Let $(\bu,\bomega,p)\in\HCUO^d\times\LO^{d(d-1)/2}\times\LOO$ and
$(\bu_h,\bomega_h,p_h)\in V_h\times W_h^i\times Q_h$ be the unique
solutions to the continuous and discrete problems \eqref{probform2} and
\eqref{probform2d}, respectively.  Assume that
$\bu\in\HusO^d$, $\bomega\in\HsO^{d(d-1)/2}$ and $p\in\HsO$, for some
$s\in(1/2,k+1]$. Then, there exists $\hat{C}>0$, independent of $h$, such
  that
\[
  \Vert(\bu,\bomega)-(\bu_h,\bomega_h)\Vert_{{X}}
+\Vert p-p_h\Vert_{0,\O}  \leq\hat{C}h^s(\Vert\bu\Vert_{1+s,\O}
+\|\bomega\|_{s,\O}+\Vert p\Vert_{s,\O}).\]
\end{theorem}

\noindent\textbf{MINI-element--$\mathbb{P}_{k}$.}
Then we consider 
the so-called MINI-element for velocity and
pressure \cite{arnold84}, and continuous or discontinuous
piecewise polynomials for vorticity 
(see  \cite[Sections 8.6 and 8.7]{bbf-2013} for further details):
\begin{align*}
H_h:&=\{\bv_h \in C(\overline{\Omega})^d:\bv_h\vert_{K}\in \mathbb{P}_{k}(K)^d \quad \forall K \in \mathcal{T}_h\}, \quad    U_h :=H_h^d, \\
\mathbb{B}(b_{K} \nabla H_h):&=\{\bv_{hb} \in H^1(\O)^d:\bv_{hb}\vert_{K}=b_K \nabla (q_h)\vert_{K} \, \text{for some } \, q_h\in H_h\},
\end{align*}
where $b_{K}$ is the standard (cubic or quartic) bubble function
$\lambda_1\cdots\lambda_{d+1}\in\mathbb{P}_{d+1}(K)$,
and let us define the following finite element subspaces:
\begin{align}\label{set2}
Q_h:&=\{ q_h \in C(\overline{\Omega}): q_h\vert_{K}\in \mathbb{P}_{k}(K)
\quad \forall K \in \mathcal{T}_h\} \cap \LOO,\quad 
V_h:=U_h \oplus \mathbb{B}(b_{K} \nabla Q_h)  \cap \HCUO^d,\nonumber\\
W_h^1:&=\{\btheta_h \in C(\overline{\Omega})^{d(d-1)/2}:
\btheta_h\vert_{K}\in \mathbb{P}_{k}(K)^{d(d-1)/2} \quad \forall K \in \mathcal{T}_h \},\\
W_h^2:&=\{\btheta_h \in \LO^{d(d-1)/2}:\btheta_h\vert_{K}\in\mathbb{P}_{k}(K)^{d(d-1)/2} \quad \forall K \in \mathcal{T}_h \}.\nonumber
\end{align}
The rate of convergence  considering the above discrete
spaces \eqref{set2} is obtained similarly as before, from the C\'ea estimate and the approximation properties. 
\begin{theorem}\label{teoMINI}
Let $k\ge1$ be a integer and let $V_h,Q_h$ and $W^i_h$, $i=1,2$
be given by \eqref{set2}.
Let $(\bu,\bomega,p)\in\HCUO^d\times\LO^{d(d-1)/2}\times\LOO$ and
$(\bu_h,\bomega_h,p_h)\in V_h\times W_h^i\times Q_h$ be the unique
solutions to the continuous and discrete problems \eqref{probform2} and
\eqref{probform2d}, respectively.  Assume that
$\bu\in\HusO^d$, $\bomega\in\HsO^{d(d-1)/2}$ and $p\in\HsO$, for some
$s\in(1/2,k]$. Then, there exists $\hat{C}>0$, independent of $h$, such
  that  
\begin{align*}
  \Vert(\bu,\bomega)-(\bu_h,\bomega_h)\Vert_{{X}}
+\Vert p-p_h\Vert_{0,\O}  \leq\hat{C}h^s(\Vert\bu\Vert_{1+s,\O}
+\|\bomega\|_{s,\O}+\Vert p\Vert_{s,\O}).\end{align*}
\end{theorem}

\noindent\textbf{Bernardi--Raugel--$\mathbb{P}_{\ric{1}}$.}
Finally we specify a family of finite elements   based on
the   Bernardi--Raugel element for velocity and
pressure \cite{BR85},  and continuous or discontinuous
piecewise linear polynomials for vorticity. 
Let us introduce the following local space of order \ric{1}
$$\mathbb{P}_{1\mathbf{n}}(K)^d:=\P_1(K)^d\oplus\textrm{span}\{\mathbf{w}_1,
\mathbf{w}_2,\mathbf{w}_3,\mathbf{w}_4\},$$
with the vector-valued functions
$\mathbf{w}_i:=\lambda_j\lambda_k\lambda_l\mathbf{n}_i\in\mathbb{P}_{3}(K)^d$,
$j,k,l\ne i$, $j,k\ne l$, $j\ne k$, with $\mathbf{n}_i$
the outer normal to the face $i$. Using the subspaces
\begin{align}\label{set3}
\nonumber V_h:&=\{\bv_h \in C(\overline{\Omega})^d:\bv_h\vert_K\in\mathbb{P}_{1\mathbf{n}}(K)^d \quad \forall K \in \mathcal{T}_h\} \cap \HCUO^d,\\
\nonumber Q_h:&=\{ q_h \in \LO: q_h\vert_{K}\in \mathbb{P}_{0}(K)
\quad \forall K \in \mathcal{T}_h\} \cap \LOO,\\
\nonumber W_h^1:&=\{\btheta_h \in C(\overline{\Omega})^{d(d-1)/2}:
\btheta_h\vert_{K}\in \mathbb{P}_{1}(K)^{d(d-1)/2} \quad \forall K \in \mathcal{T}_h \},\\
W_h^2:&=\{\btheta_h \in \LO^{d(d-1)/2}:\btheta_h\vert_{K}\in\mathbb{P}_{\ric{1}}(K)^{d(d-1)/2} \quad \forall K \in \mathcal{T}_h \},
\end{align}
the rate of convergence of the augmented mixed
finite element scheme   is as follows: 
\begin{theorem}\label{teobernardi}
Let $V_h,Q_h$ and $W^i_h$, $i=1,2$
be given by \eqref{set3}.
Let $(\bu,\bomega,p)\in\HCUO^d\times\LO^{d(d-1)/2}\times\LOO$ and
$(\bu_h,\bomega_h,p_h)\in V_h\times W_h^i\times Q_h$ be the unique
solutions to the continuous and discrete problems \eqref{probform2} and
\eqref{probform2d}, respectively.  Assume that
$\bu\in\HusO^d$, $\bomega\in\HsO^{d(d-1)/2}$ and $p\in\HsO$, for some
$s\in(1/2,1]$. Then, there exists $\hat{C}>0$, independent of $h$, such
  that  
\begin{align*}
&\Vert(\bu,\bomega)-(\bu_h,\bomega_h)\Vert_{{X}}
+\Vert p-p_h\Vert_{0,\O}  \leq\hat{C}h^s(\Vert\bu\Vert_{1+s,\Omega}
+\|\bomega\|_{s,\Omega}+\Vert p\Vert_{s,\Omega}).
\end{align*}

\end{theorem}

\section{Numerical results}\label{sec:results}
In this section, we present some numerical experiments that serve to verify numerically the convergence rates predicted by Theorems \ref{teoTaylorHood}, \ref{teoMINI}, and that illustrate the performance of the proposed methods in typical incompressible flow problems. We implement the finite element routines using the library FEniCS \cite{alnaes15}. The nonlinear systems are solved with a Newton--Raphson method with zero initial guess, and prescribing a tolerance of $10^{-8}$ on the either  absolute or relative $\ell^\infty$ norm of the residuals. Unless otherwise specified, all linear solves performed at each nonlinear iteration step are conducted with the direct solver MUMPS \cite{mumps}.

\begin{table}[t!]
\begin{center}
\caption{Example 1A: Accuracy verification in 2D for different finite element families for the approximation of velocity and pressure, whereas for the vorticity, discontinuous $\mathbb{P}_1$ elements (the space $W_h^2$) are used in all cases.}\label{table:2d}
\begin{tabular*}{\textwidth}{@{\extracolsep{\fill}}rccccccc}
\hline
DoF & $h$&$\nnorm{\bu-\bu_h}_{1,\O}$& $r(\bu)$ &$\Vert \bomega-\bomega_h\Vert_{0,\O}$&$r(\bomega)$& $\Vert p-p_h\Vert_{0,\O}$ & $r(p)$\cr
\hline
\multicolumn{8}{c}{Taylor--Hood element}\cr
\hline 
    84 & 0.707 & 8.52e-01 & -- & 5.44e-01 & -- & 2.33e-01 & -- \\
   284 & 0.354 & 2.49e-01 & 1.772 & 1.41e-01 & 1.948 & 4.64e-02 & 2.326\\
  1044 & 0.177 & 5.78e-02 & 2.110 & 3.35e-02 & 2.073 & 7.38e-03 & 2.651\\
  4004 & 0.088 & 1.29e-02 & 2.163 & 8.21e-03 & 2.028 & 1.67e-03 & 2.148\\
 15684 & 0.044 & 3.05e-03 & 2.081 & 2.04e-03 & 2.008 & 4.06e-04 & 2.038\\
 62084 & 0.022 & 7.50e-04 & 2.024 & 5.09e-04 & 2.003 & 1.01e-04 & 2.010\\
247044 & 0.011 & 1.87e-04 & 2.006 & 1.27e-04 & 2.001 & 2.51e-05 & 2.003\\
\hline
\multicolumn{8}{c}{MINI-element}\\
\hline 
     68 & 0.707 & 2.58e+0 & -- & 1.05e+0 & -- & 4.26e-01 & -- \\
   236 & 0.354 & 1.53e+0 & 0.760 & 4.40e-01 & 1.261 & 9.12e-02 & 2.224\\
   884 & 0.177 & 7.69e-01 & 0.988 & 2.16e-01 & 1.024 & 2.30e-02 & 1.989\\
  3428 & 0.088 & 3.83e-01 & 1.007 & 1.07e-01 & 1.020 & 5.71e-03 & 2.009\\
 13508 & 0.044 & 1.91e-01 & 1.003 & 5.30e-02 & 1.008 & 1.51e-03 & 1.917\\
 53636 & 0.022 & 9.55e-02 & 1.001 & 2.65e-02 & 1.002 & 4.19e-04 & 1.851\\
213764 & 0.011 & 4.77e-02 & 1.000 & 1.32e-02 & 1.000 & 1.22e-04 & 1.777\\
\hline
\multicolumn{8}{c}{Bernardi--Raugel element}\\
\hline 
66 & 0.707 & 1.07e+0 & -- & 8.26e-01 & -- & 2.79e-01 & --\\
234 & 0.354 & 5.14e-01 & 1.067 & 3.46e-01 & 1.252 & 1.50e-01 & 0.892 \\
882 & 0.177 & 2.70e-01 & 0.928 & 1.81e-01 & 0.931 & 7.15e-02 & 1.075 \\
3426 & 0.088 & 1.40e-01 & 0.947 & 9.58e-02 & 0.924 & 3.41e-02 & 1.066 \\
13506 & 0.044 & 7.08e-02 & 0.984 & 4.86e-02 & 0.977 & 1.67e-02 & 1.026 \\
53634 & 0.022 & 3.55e-02 & 0.995 & 2.44e-02 & 0.992 & 8.33e-03 & 1.008 \\
213762 & 0.011 & 1.77e-02 & 0.998 & 1.22e-02 & 0.997 & 4.16e-03 & 1.002 \\
\hline
\multicolumn{8}{c}{Crouzeix--Raviart element}\\
\hline 
    65 & 0.707 & 2.33e+0 & -- & 1.92e+0 & -- & 4.39e-01 & -- \\
   241 & 0.354 & 2.19e+0 & 0.084 & 1.14e+0 & 0.754 & 4.02e-01 & 0.126\\
   929 & 0.177 & 1.38e+0 & 0.748 & 5.47e-01 & 1.057 & 2.25e-01 & 0.962\\
  3649 & 0.088 & 6.60e-01 & 1.069 & 2.35e-01 & 1.218 & 9.02e-02 & 1.321\\
 14465 & 0.044 & 3.21e-01 & 1.038 & 1.10e-01 & 1.091 & 4.12e-02 & 1.131\\
 57601 & 0.022 & 1.58e-01 & 1.024 & 5.42e-02 & 1.026 & 2.00e-02 & 1.039\\
229889 & 0.011 & 7.82e-02 & 1.014 & 2.69e-02 & 1.007 & 9.95e-03 & 1.010\\
\hline
\end{tabular*}
\end{center}
\end{table}

\begin{table}[t!]
\begin{center}
\caption{Example 1A: Effect of the augmentation constants on the convergence properties. Here we only use Taylor--Hood type of elements (compare with Table~\ref{table:2d}, top rows) and discontinuous vorticity approximation, $W_h^2$ with $k = 1$.}\label{table:2d-kappa}
\begin{tabular*}{\textwidth}{@{\extracolsep{\fill}}rccccccc}
\hline
DoF & $h$&$\nnorm{\bu-\bu_h}_{1,\O}$& $r(\bu)$ &$\Vert \bomega-\bomega_h\Vert_{0,\O}$&$r(\bomega)$& $\Vert p-p_h\Vert_{0,\O}$ & $r(p)$\cr
\hline
\multicolumn{8}{c}{$\kappa_1 = \kappa_2 = 0$}\\
\hline 
    84 & 0.707 & 1.01e+00 & -- & 5.52e-01 & -- & 2.34e-01 & --\\
   284 & 0.354 & 3.99e-01 & 1.333 & 1.52e-01 & 1.860 & 5.43e-02 & 2.109\\
  1044 & 0.177 & 1.71e-01 & 1.224 & 3.75e-02 & 2.020 & 9.12e-03 & 2.573\\
  4004 & 0.088 & 7.93e-02 & 1.108 & 9.25e-03 & 2.018 & 2.07e-03 & 2.140\\
 15684 & 0.044 & 4.01e-02 & 0.986 & 2.30e-03 & 2.007 & 5.04e-04 & 2.038\\
 62084 & 0.022 & 2.00e-02 & 1.002 & 5.72e-04 & 2.008 & 1.25e-04 & 2.015\\
 247044 & 0.011 & 1.00e-02 & 0.995 & 1.42e-04 & 2.006 & 3.11e-05 & 2.004\\
 \hline
\multicolumn{8}{c}{$\kappa_1 = 2/3 \nu_0$, $\kappa_2= 0$}\\
\hline 
     84 & 0.707 & 1.01e+00 & -- & 5.52e-01 & -- & 2.34e-01 & --\\
   284 & 0.354 & 3.99e-01 & 1.333 & 1.52e-01 & 1.860 & 5.43e-02 & 2.109\\
  1044 & 0.177 & 1.71e-01 & 1.224 & 3.75e-02 & 2.020 & 9.12e-03 & 2.573\\
  4004 & 0.088 & 7.93e-02 & 1.108 & 9.25e-03 & 2.018 & 2.07e-03 & 2.140\\
 15684 & 0.044 & 4.01e-02 & 0.986 & 2.30e-03 & 2.007 & 5.04e-04 & 2.038\\
 62084 & 0.022 & 2.00e-02 & 1.002 & 5.72e-04 & 2.008 & 1.25e-04 & 2.015\\
247044 & 0.011 & 1.00e-02 & 0.995 & 1.42e-04 & 2.006 & 3.11e-05 & 2.004\\
 \hline
\multicolumn{8}{c}{$\kappa_1 = 0$, $\kappa_2= 0.5\nu_0$}\\
\hline
    84 & 0.707 & 8.50e-01 & -- & 5.44e-01 & -- & 2.33e-01 & --\\
   284 & 0.354 & 2.50e-01 & 1.766 & 1.41e-01 & 1.948 & 4.64e-02 & 2.326\\
  1044 & 0.177 & 5.80e-02 & 2.107 & 3.35e-02 & 2.073 & 7.38e-03 & 2.651\\
  4004 & 0.088 & 1.29e-02 & 2.164 & 8.21e-03 & 2.028 & 1.67e-03 & 2.148\\
 15684 & 0.044 & 3.06e-03 & 2.082 & 2.04e-03 & 2.008 & 4.06e-04 & 2.038\\
 62084 & 0.022 & 7.51e-04 & 2.026 & 5.09e-04 & 2.003 & 1.01e-04 & 2.010\\
247044 & 0.011 & 1.87e-04 & 2.007 & 1.27e-04 & 2.001 & 2.51e-05 & 2.003\\
\hline
\multicolumn{8}{c}{$\kappa_1 = 2/3 \nu_0$, $\kappa_2 = 0.1\nu_0$}\\
\hline 
    84 & 0.707 & 9.50e-01 & -- & 5.49e-01 & -- & 2.31e-01 & --\\
   284 & 0.354 & 3.38e-01 & 1.490 & 1.47e-01 & 1.897 & 5.24e-02 & 2.141\\
  1044 & 0.177 & 1.00e-01 & 1.757 & 3.45e-02 & 2.092 & 8.01e-03 & 2.711\\
  4004 & 0.088 & 2.31e-02 & 2.117 & 8.26e-03 & 2.065 & 1.68e-03 & 2.255\\
 15684 & 0.044 & 4.36e-03 & 2.403 & 2.04e-03 & 2.018 & 4.06e-04 & 2.048\\
 62084 & 0.022 & 8.63e-04 & 2.336 & 5.09e-04 & 2.003 & 1.01e-04 & 2.011\\
247044 & 0.011 & 1.95e-04 & 2.148 & 1.27e-04 & 2.001 & 2.51e-05 & 2.003\\
\hline
\multicolumn{8}{c}{$\kappa_1 = 2/3 \nu_0$, $\kappa_2 = 10\nu_0$}\\
\hline
    84 & 0.707 & 7.49e-01 & -- & 5.87e-01 & -- & 3.98e-01 & --\\
   284 & 0.354 & 1.93e-01 & 1.958 & 1.53e-01 & 1.940 & 4.84e-02 & 3.039\\
  1044 & 0.177 & 4.83e-02 & 1.997 & 3.60e-02 & 2.086 & 7.85e-03 & 2.625\\
  4004 & 0.088 & 1.20e-02 & 2.013 & 8.48e-03 & 2.088 & 1.69e-03 & 2.214\\
 15684 & 0.044 & 2.98e-03 & 2.004 & 2.06e-03 & 2.039 & 4.07e-04 & 2.055\\
 62084 & 0.022 & 7.45e-04 & 2.000 & 5.11e-04 & 2.013 & 1.01e-04 & 2.014\\
247044 & 0.011 & 1.86e-04 & 1.999 & 1.27e-04 & 2.004 & 2.51e-05 & 2.004\\
\hline
\end{tabular*}
\end{center}
\end{table}

\begin{table}[t!]
\begin{center}
\caption{Example 1B: Effect of the augmentation constants on the convergence properties. Here we only use Taylor--Hood type of elements (compare with Table~\ref{table:2d}, top rows) and continuous vorticity approximation, $W_h^2$ with $k = 1$.}\label{table:2d-kappa-cont}
\begin{tabular*}{\textwidth}{@{\extracolsep{\fill}}rccccccc}
\hline
DoF & $h$&$\nnorm{\bu-\bu_h}_{1,\O}$& $r(\bu)$ &$\Vert \bomega-\bomega_h\Vert_{0,\O}$&$r(\bomega)$& $\Vert p-p_h\Vert_{0,\O}$ & $r(p)$\\
 \hline
\multicolumn{8}{c}{$\kappa_1 = 2/3 \nu_0$, $\kappa_2= 0$}\\
\hline 
    69 & 0.707 & 3.84e+00 & -- & 9.57e-01 & -- & 1.91e-01 & -- \\
   213 & 0.354 & 2.37e+00 & 1.249 & 3.19e-01 & 1.584 & 1.23e-01 & 0.641\\
   741 & 0.177 & 1.93e+00 & 1.182 & 5.53e-02 & 2.528 & 3.86e-02 & 1.670\\
  2757 & 0.088 & 4.43e-01 & 2.119 & 1.08e-02 & 2.361 & 3.50e-03 & 3.461\\
 10629 & 0.044 & 1.97e-01 & 1.169 & 2.55e-03 & 2.080 & 6.02e-04 & 2.540\\
 41733 & 0.022 & 9.88e-02 & 0.997 & 6.33e-04 & 2.010 & 1.35e-04 & 2.163\\
165381 & 0.011 & 4.96e-02 & 0.995 & 1.58e-04 & 2.002 & 3.17e-05 & 2.084\\
 \hline
\multicolumn{8}{c}{$\kappa_1 = 0$, $\kappa_2= 0.5\nu_0$}\\
\hline
     69 & 0.707 & 3.26e+00 & -- & 9.05e-01 & -- & 1.31e-01 & --\\
   213 & 0.354 & 2.58e+00 & 1.106 & 3.29e-01 & 1.458 & 1.55e-01 & -0.240\\
   741 & 0.177 & 2.14e+00 & 0.238 & 7.69e-02 & 2.098 & 3.46e-02 & 2.159\\
  2757 & 0.088 & 1.03e+00 & 1.054 & 1.40e-02 & 2.459 & 4.71e-03 & 2.877\\
 10629 & 0.044 & 5.12e-01 & 1.012 & 3.19e-03 & 2.131 & 8.97e-04 & 2.392\\
 41733 & 0.022 & 2.66e-01 & 0.945 & 7.96e-04 & 2.005 & 3.40e-04 & 1.400\\
165381 & 0.011 & 1.90e-01 & 0.483 & 2.33e-04 & 1.771 & 2.39e-04 & 0.510\\
\hline
\multicolumn{8}{c}{$\kappa_1 = 2/3 \nu_0$, $\kappa_2 = 0.1\nu_0$}\\
\hline 
    69 & 0.707 & 3.61e+00 & -- & 9.38e-01 & -- & 1.66e-01 & -- \\
   213 & 0.354 & 2.92e+00 & 0.357 & 2.95e-01 & 1.667 & 1.25e-01 & 0.413\\
   741 & 0.177 & 9.29e-01 & 1.651 & 5.16e-02 & 2.515 & 2.26e-02 & 2.467\\
  2757 & 0.088 & 1.55e-01 & 2.582 & 1.05e-02 & 2.304 & 1.94e-03 & 3.539\\
 10629 & 0.044 & 2.18e-02 & 2.829 & 2.53e-03 & 2.046 & 4.08e-04 & 2.249\\
 41733 & 0.022 & 2.89e-03 & 2.918 & 6.31e-04 & 2.004 & 1.01e-04 & 2.020\\
165381 & 0.011 & 3.99e-04 & 2.856 & 1.58e-04 & 2.001 & 2.51e-05 & 2.003\\
\hline
\multicolumn{8}{c}{$\kappa_1 = 2/3 \nu_0$, $\kappa_2 = 10\nu_0$}\\
\hline
    69 & 0.707 & 1.65e+00 & -- & 8.60e-01 & -- & 2.56e-01 & -- \\
   213 & 0.354 & 1.08e+00 & 0.391 & 2.83e-01 & 1.604 & 6.89e-02 & 1.894\\
   741 & 0.177 & 4.64e-01 & 1.220 & 4.73e-02 & 2.582 & 9.53e-03 & 2.854\\
  2757 & 0.088 & 1.22e-01 & 1.927 & 1.04e-02 & 2.184 & 1.74e-03 & 2.456\\
 10629 & 0.044 & 2.25e-02 & 2.437 & 2.53e-03 & 2.037 & 4.07e-04 & 2.093\\
 41733 & 0.022 & 3.37e-03 & 2.741 & 6.31e-04 & 2.004 & 1.01e-04 & 2.015\\
165381 & 0.011 & 4.77e-04 & 2.820 & 1.58e-04 & 2.001 & 2.51e-05 & 2.003\\
\hline
\end{tabular*}
\end{center}
\end{table}

\subsection{Example 1: Convergence tests in 2D and 3D}
To numerically investigate the accuracy of the proposed finite element formulation, we consider the unit square and unit cube domains $\O = (0,1)^2$, $\O = (0,1)^3$ discretised into uniform triangular/tetrahedral elements. Sequences of successively refined meshes are constructed and the numerical solutions obtained on each refinement level are compared against manufactured exact solutions in the 2D and 3D cases, defined respectively as 
\begin{gather*}
\bu(x,y)=\begin{pmatrix}\cos(\pi x)\sin(\pi y)\\
-\sin(\pi x)\cos(\pi y)\end{pmatrix}, \quad 
\bu(x,y,z) = \begin{pmatrix}\sin(\pi x)\cos(\pi y)\cos(\pi z)\\ -2\cos(\pi x)\sin(\pi y)\cos(\pi z)\\ \cos(\pi x)\cos(\pi y)\sin(\pi z)\end{pmatrix}, \\
p(x,y) = \sin(\pi x)\sin(\pi y), \quad p(x,y,z)=1-\cos(xyz)\sin(xyz),
\end{gather*}
with $\bomega = \curl\bu$, with permeability $\kappa = 0.1$, with the following distributions of 2D and 3D viscosities
\[\nu(x,y) = \nu_0 + (\nu_1 - \nu_0) \cos(\pi xy)\cos(\pi xy), \quad 
\nu(x,y,z) = \nu_0 + (\nu_1 - \nu_0) x^2y^2z^2,\]
and with $\sigma(\bx) = \kappa^{-1}\nu(\bx)$. The velocities are divergence-free. Moreover, the viscosities are smooth, uniformly bounded, and characterised by the parameters $\nu_0 = 0.1$, $\nu_1 = 1$. Note that more realistic viscosity profiles might exhibit high gradients and variations that are typically non-smooth (e.g. random distributions) and in such cases taking gradients becomes prohibitive. 
The augmentation constants assume the values $\kappa_1 = \frac{2}{3}\nu_0$ and $\kappa_2 = \frac12\nu_0$. Since the exact velocity is enforced essentially everywhere on the boundary, the mean value of pressure is to be fixed (to match that of the exact pressure). This is done with a real Lagrange multiplier. Moreover, the source term $\ff$ is constructed from the momentum balance equation so that the closed-form solutions above become an exact solution of the problem.

\begin{table}[t!]
\begin{center}
\caption{Example 1C: Error profiles and experimental rates of convergence for velocity, vorticity, and pressure  approximation in 3D generated with different finite element families.}\label{table:3d}
\begin{tabular*}{\textwidth}{@{\extracolsep{\fill}}rccccccc}
\hline
DoF & $h$&$\nnorm{\bu-\bu_h}_{1,\O}$& $r(\bu)$ &$\Vert \bomega-\bomega_h\Vert_{0,\O}$&$r(\bomega)$& $\Vert p-p_h\Vert_{0,\O}$ & $r(p)$\\ 
\hline
\multicolumn{8}{c}{Taylor--Hood and $\mathbb{P}_1$ ($W_h^1$)}\\
\hline 
   484 & 0.866 & 1.43e+0 & -- & 1.14e+0 & -- & 1.28e-01 & --\\
  2688 & 0.433 & 3.78e-01 & 1.925 & 3.20e-01 & 1.833 & 1.41e-02 & 3.189\\
 17656 & 0.217 & 9.57e-02 & 1.982 & 6.85e-02 & 2.223 & 1.61e-03 & 3.130\\
127464 & 0.108 & 2.32e-02 & 2.047 & 1.62e-02 & 2.080 & 2.26e-04 & 2.834\\
967624 & 0.054 & 5.60e-03 & 1.929 & 3.99e-03 & 2.095 & 5.36e-05 & 2.387\\
\hline
\multicolumn{8}{c}{MINI-element and $\mathbb{P}_1$ ($W_h^1$)}\\
\hline 
    334 & 0.866 & 5.29e+0 & -- & 1.78e+0 &-- & 7.75e-01 & --\\
  2028 & 0.433 & 2.93e+0 & 0.852 & 7.23e-01 & 1.301 & 4.18e-01 & 0.890\\
 14320 & 0.217 & 1.29e+0 & 1.189 & 2.22e-01 & 1.703 & 1.10e-01 & 1.926\\
108120 & 0.108 & 6.05e-01 & 1.089 & 6.45e-02 & 1.785 & 2.93e-02 & 1.910\\
841384 & 0.054 & 2.96e-01 & 1.031 & 1.90e-02 & 1.765 & 8.21e-03 & 1.835\\
\hline
\multicolumn{8}{c}{Crouzeix--Raviart and $\mathbb{P}_0$ ($W_h^2$ with $k = 0$)}\\
\hline 
   553 & 0.866 & 2.84e-01 & -- & 1.92e-01 & -- & 7.72e-02 & --\\
  4129 & 0.433 & 1.39e-01 & 1.028 & 7.59e-02 & 1.343 & 3.62e-02 & 1.094\\
 31873 & 0.217 & 4.58e-02 & 1.602 & 2.48e-02 & 1.616 & 1.61e-02 & 1.167\\
250369 & 0.108 & 1.34e-02 & 1.774 & 7.96e-03 & 1.637 & 7.85e-03 & 1.036\\
1984513& 0.054 & 4.33e-03 & 1.588 & 2.81e-03 & 1.609 & 3.82e-03 & 1.121\\
\hline
\end{tabular*}
\end{center}
\end{table}

Errors between exact and approximate numerical solutions in the individual norms for velocity, vorticity and pressure are shown with respect to the number of degrees of freedom (DoF) in Tables~\ref{table:2d} and \ref{table:3d} for the 2D and 3D cases, respectively. These errors were produced with the use of Taylor--Hood \cite{HT}, MINI-element \cite{arnold84}, Bernardi--Raugel elements \cite{BR85}, and for sake of illustration, for 2D we include as well the non-conforming Crouzeix--Raviart elements with facet stabilisation \cite{crouzeix73}. We also tabulate the local experimental rate of error decay for the generic vector or scalar field $s$, computed as  $r(s) =\log({\tt e}(s)/\widehat{{\tt e}}(s))[\log(h/\widehat{h})]^{-1}$,
where ${\tt e}$ and $\widehat{{\tt e}}$ denote errors produced on two consecutive
meshes associated with mesh sizes $h$ and $\widehat{h}$, respectively. 
 In all instances we observe a monotonic optimal convergence, as supported by Theorems~\ref{teoTaylorHood}, \ref{teoMINI}, \ref{teobernardi}. For MINI-elements we observe a higher convergence than the one anticipated for pressure. For these tests the number of Newton--Raphson steps required to achieve the prescribed tolerance was, in average, 3. Similar results as those reported herein were obtained by replacing the viscosities described above, with profiles exhibiting higher gradients. 

 \ric{We also explore the  influence of the augmentation constants $\kappa_1,\kappa_2$ on the convergence properties, and the results are collected in Table~\ref{table:2d-kappa}. The second block in the table shows that for $\kappa_1 =0$ the convergence is not affected. 
 The first and third blocks in the table indicate that this is needed to recover the optimal convergence of the velocity approximation. 

We also note that for these cases in Table~\ref{table:2d-kappa} we use a discontinuous approximation for vorticity, $W_h^2$ with $k = 1$. The need of $\kappa_1$ (which we recall it enforces that $\bomega_h = \curl\bu_h$) is perhaps better exemplified in Table~\ref{table:2d-kappa-cont}. Using a continuous vorticity approximation (with   $W_h^1$ and $k = 1$), a stabilisation constant $\kappa_1$ satisfying the bounds specified in Lemma \eqref{lem-elip} produces optimal convergence rates, whereas in contrast with Table~\ref{table:2d-kappa}, the value $\kappa_1 = 0$ makes the convergence to deteriorate. }

\begin{figure}[t!]
\begin{center}
\includegraphics[width = 0.475\textwidth]{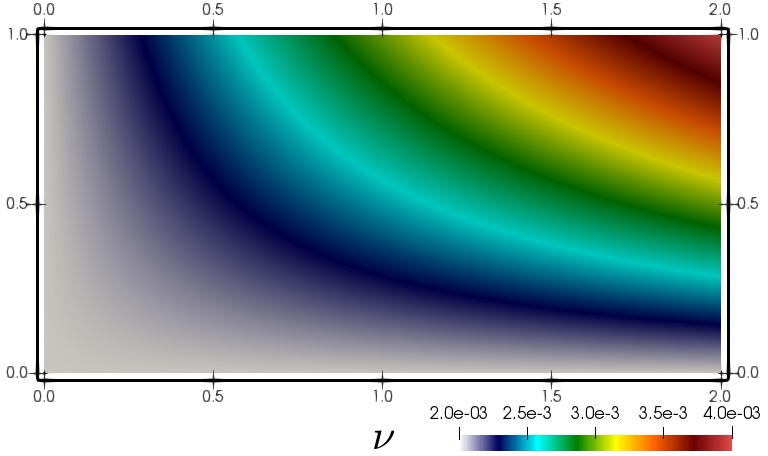}
\includegraphics[width = 0.475\textwidth]{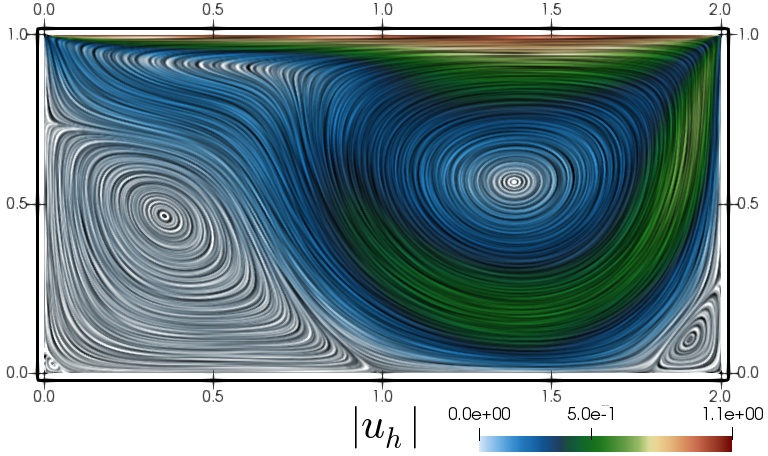}\\
\includegraphics[width = 0.475\textwidth]{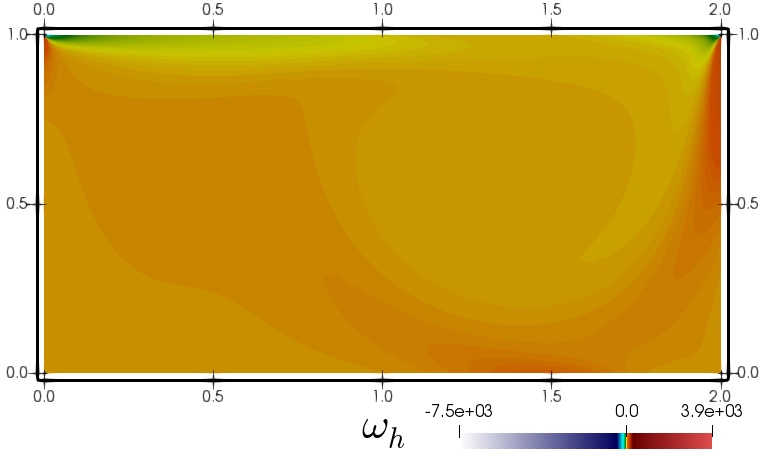}
\includegraphics[width = 0.475\textwidth]{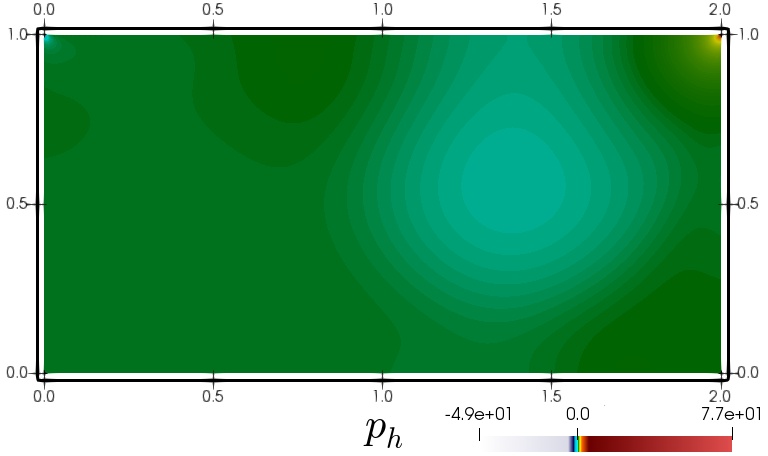}\\
\includegraphics[width = 0.475\textwidth]{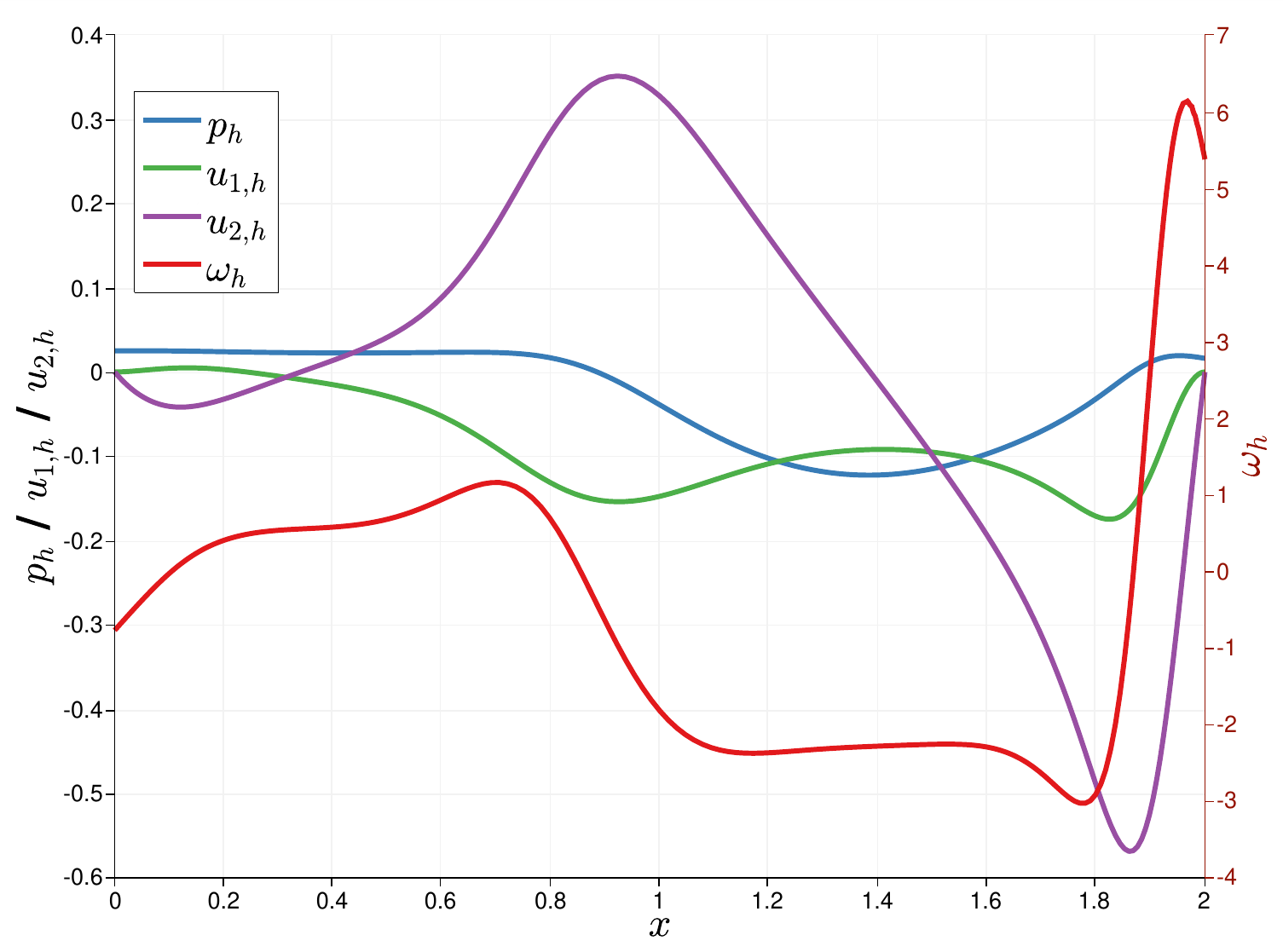}
\includegraphics[width = 0.475\textwidth]{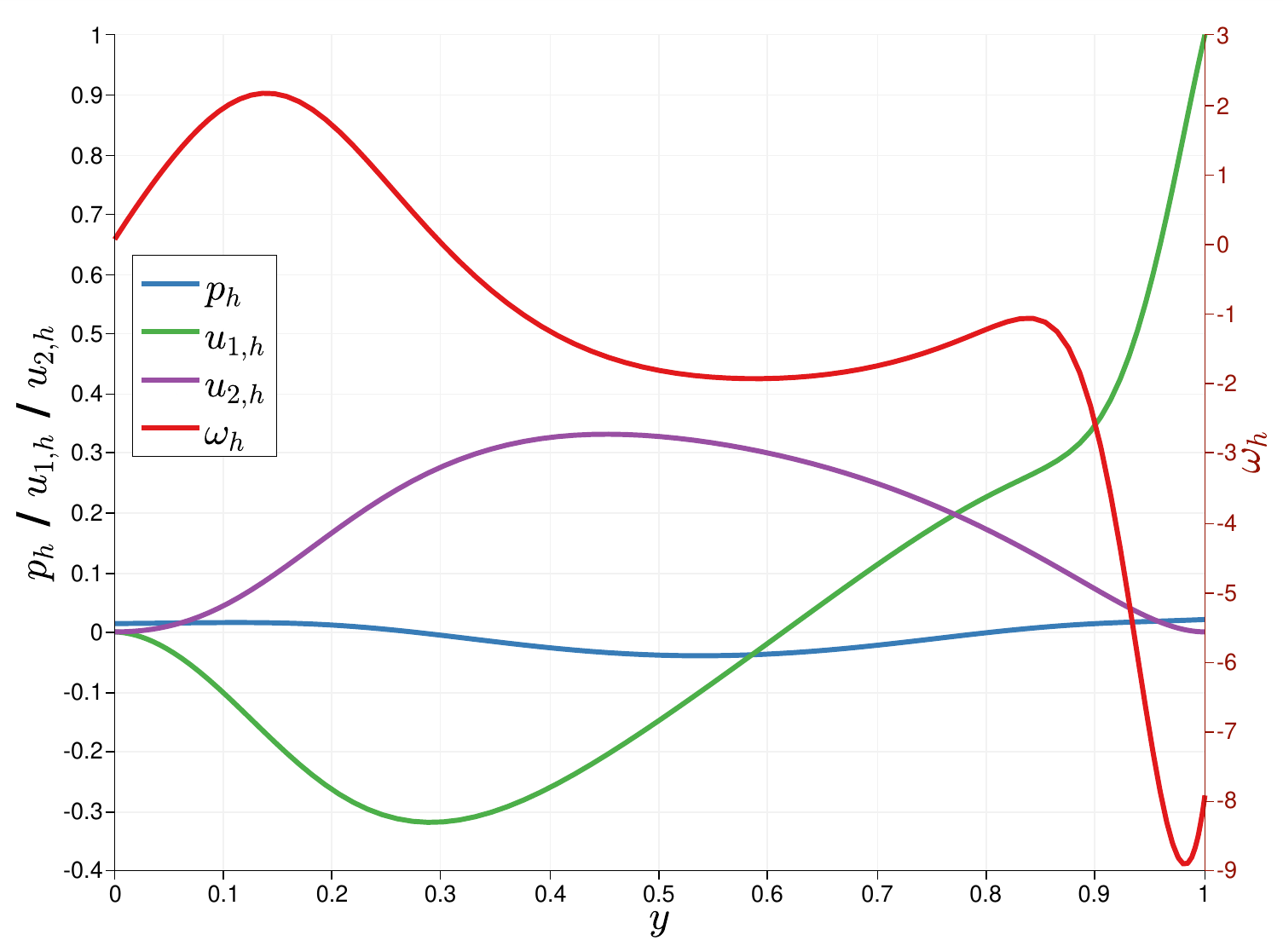}
\end{center}
\caption{Example 2: Lid-driven wide cavity. Distribution of variable kinematic viscosity, and portrait of approximate solutions (velocity magnitude with velocity line integral convolution paths, scalar vorticity, and pressure profile). The bottom panels show cuts on the horizontal ($y=0.5$) and vertical ($x=1$) midlines of the cavity along with solution profiles.}\label{fig:ex02}
\end{figure} 

\subsection{Example 2: Lid-driven wide cavity flow}
Consider now the modified lid-driven cavity benchmark problem in the case of a wide two-dimensional domain $\Omega = (0,2)\times(0,1)$. The mesh is structured and with $200\times100$ vertices. The grid is stretched to cluster towards the boundaries. We use in this case the MINI-element plus $W_h^1$ space for vorticity. The viscosity is $\nu = \nu_0(1 + \frac12xy)$ with $\nu_0 = 0.002$, and the external force is zero $\ff = \cero$.  The top wall is moving to the right with constant velocity $\bu = (1,0)^{\tt t}$ while   the other sub-boundaries are equipped with a no-slip velocity condition, implying that a discontinuity exists in the Dirichlet datum. Again the mean-value of the pressure is fixed (to be zero) with a real Lagrange multiplier. We stress that, even in the case of constant viscosity, the exact solutions for this classical problem are not known. 
We report the obtained approximate flow patterns in Figure~\ref{fig:ex02}. The first row shows the imposed viscosity profile (which also defines a variable Brinkman parameter $\sigma$), the velocity field with line integral convolution to better visualise the velocity recirculation and vortices (which differs from the usual square lid-driven cavity  flow) and indicating well-resolved flow patterns, and the middle row displays scalar vorticity and pressure distributions. The usual corner singularities are seen in the pressure profiles, but no spurious oscillations are observed. We also show in the bottom row plots of the solutions on the midlines $y=0.5$ and $x=1$.

\begin{figure}[t!]
\begin{center}
\includegraphics[width = 0.475\textwidth]{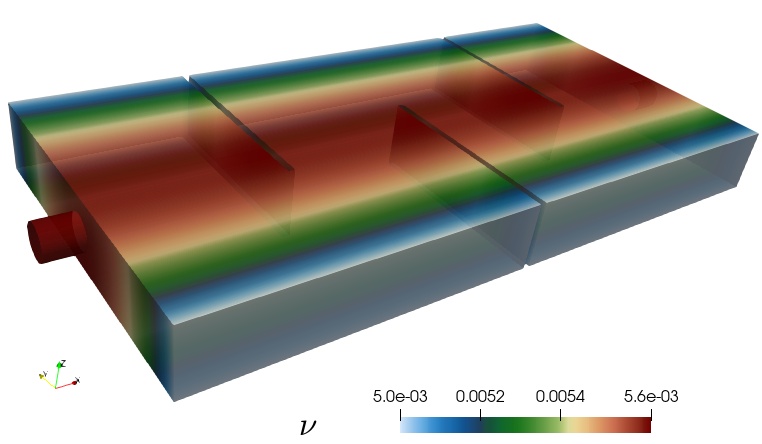}
\includegraphics[width = 0.475\textwidth]{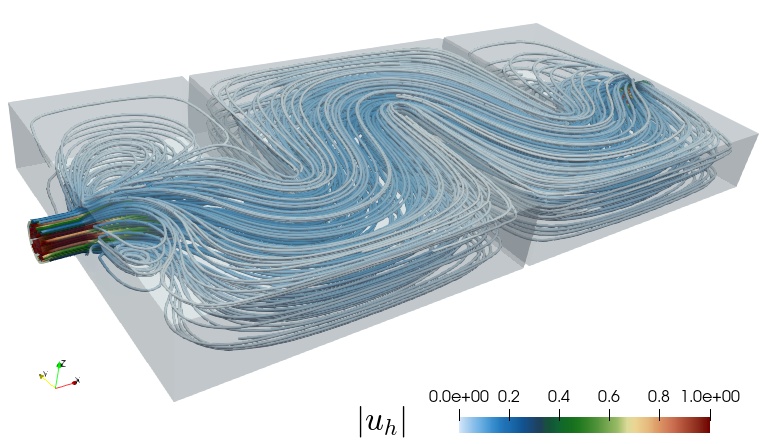}\\
\includegraphics[width = 0.475\textwidth]{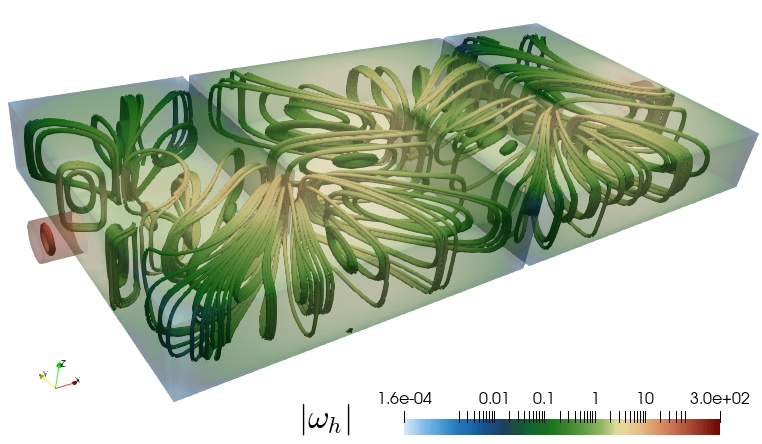}
\includegraphics[width = 0.475\textwidth]{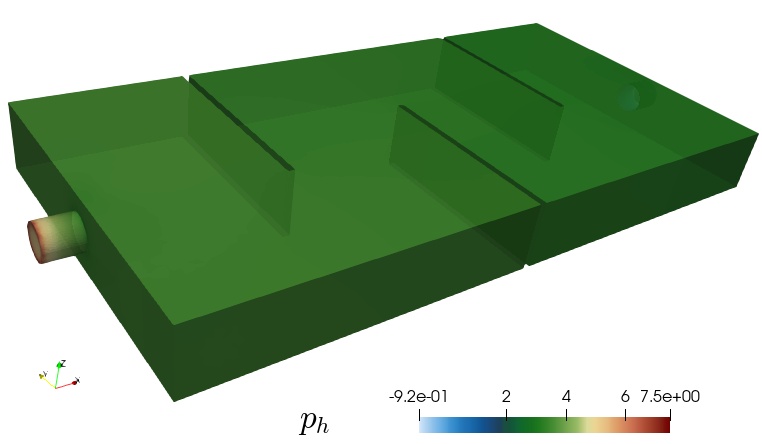}
\end{center}

\caption{Example 3: Flow in a maze-shaped geometry. Distribution of variable kinematic viscosity, and portrait of approximate solutions (velocity magnitude and velocity streamlines, vorticity magnitude and vorticity streamlines, and pressure profile).}\label{fig:ex03}
\end{figure} 
\subsection{Example 3: Flow in a maze-shaped 3D geometry}
To conclude the set of numerical examples we present the simulation of three-dimensional flow over obstacles as those encountered in mazes. 
The geometry and the mesh have been generated using the mesh manipulator Gmsh \cite{geuzaine09}. The resulting tetrahedral mesh has 193375  elements and 40418 vertices. For this example the viscosity assumes the form $\nu = nu_0 (1+ \frac12 y(1-y))$ with $\nu_0 = 0.005$. The channel has a length of 2.2 (aligned with the $x$ axis), a width of 1, and a height of 0.2. 
The boundary of the domain is partitioned into an inflow surface (a disk of radius 0.05 located on the plane $x_{\min}$), an outflow surface (the disk of the same dimensions, located at $x_{\max}$), and the remainder of the boundary is constituted by the walls of the maze. On the inflow, a constant inlet velocity is prescribed $\bu_{\text{in}} = (1,0,0)^{\tt t}$, no-slip velocities are considered on the walls, and zero pseudo-stress is imposed as outflow boundary condition. These boundary conditions do not coincide with the pure Dirichlet case analysed in the paper. As in the previous example, we consider $\ff = \cero$. Dictated by inf-sup stability requirements, we use again the MINI-element plus piecewise linear vorticity approximation in $W_h^1$. 

A graphical illustration regarding the behaviour of the channeling flow is given in Figure \ref{fig:ex03}, showing a higher viscosity near the centre of the channel.  The velocity streamlines plotted on the top-right panel clearly exemplify the flow patterns commonly found for this regime (with maximal Reynolds number of $\text{Re}=\frac{ \text{vel} \cdot \text{diam}}{\nu_0}=\frac{1\cdot 0.1}{0.005}=20$), namely a very mild recirculation near the walls adjacent to the inlet and an almost Poiseuille flow, showing higher velocity magnitude towards the centre of the channel. Also, and as expected, the pressure drops from inlet to outlet, and  it has the imposed zero mean value.


\section{Summary and concluding remarks}\label{sec:concl}
In the present work, we have presented an optimally convergent mixed finite element method for the discretisation of the vorticity-velocity-pressure formulation of the Navier--Stokes equations with non-constant viscosity. We have established the unique solvability of the continuous and discrete variational problems using the theory of fixed--point operators, and have shown the stability and optimal convergence of particular choices of finite element families. We have numerically investigated the performance of the methods on 2D and 3D tests. Some key features of the proposed method are the liberty to choose different inf-sup stable finite element families for the Navier--Stokes equations, the direct and accurate access to vorticity independently and without applying postprocessing, and the flexibility in handling Dirichlet boundary conditions for velocity, as well as in defining outflow conditions. Some of these properties are of marked interest, for instance, when coupling with other effects such as settling mechanisms or doubly-diffusive interactions, where we foresee a direct applicability of the ideas proposed herein. On the other hand, one unavoidable disadvantage of the present method is that it does not produce point-wise divergence-free velocities. \ric{Currently, we are not aware of other formulations in the literature that address the problem in the present form. Therefore a thorough comparison against other discretisations is out of the scope of the present paper.}

\bigskip 
\noindent\textbf{Acknowledgements.}
This work has been partially supported by DICREA-UBB through projects 2020127 IF/R and 2120173 GI/C,  by ANID-Chile through Centro de Modelamiento Matemático (FB210005), Anillo of Computational Mathematics for Desalination Processes (ACT 210087), Fondecyt project 1211265; by Scholarship Program Doctorado Becas Chile 2021 21210945;  by the Monash Mathematics Research Fund S05802-3951284; and by the Ministry of Science and Higher Education of the Russian Federation within the framework of state support for the creation and development of World-Class Research Centers ``Digital biodesign and personalized healthcare'' No. 075-15-2022-304.

\end{document}